\documentclass[NoFloatCountersInSection]{cedram-aif}

\usepackage[T1]{fontenc}
\usepackage[utf8]{inputenc}
\usepackage[english]{babel}
\usepackage[fixlanguage]{babelbib}
\usepackage{cite}
\usepackage{yfonts}
\usepackage{mathrsfs}
\usepackage{colonequals}
\usepackage{enumerate}
\usepackage{tikz}

\usepackage[all,pdf]{xy} 
\SelectTips{cm}{11} 

\usepackage{hyperref}

% ...

\newcommand{\cU}{\ensuremath{\mathcal{U}}}
\newcommand{\cV}{\ensuremath{\mathcal{V}}}

\newcommand{\dd}{\mathop{}\!\mathrm{d}}

\newcommand{\varep}{\varepsilon}

\newcommand{\R}{\ensuremath{\mathbb{R}}}
\newcommand{\N}{\ensuremath{\mathbb{N}}}

\DeclareMathOperator{\id}{id}
\DeclareMathOperator{\pr}{pr}
\DeclareMathOperator{\supp}{supp}
\DeclareMathOperator{\ev}{ev}
\DeclareMathOperator{\res}{res}
\DeclareMathOperator{\proj}{proj}
\DeclareMathOperator{\Tay}{Tay}
\DeclareMathOperator{\im}{im}
\newcommand{\Frechet}{Fr\'{e}chet }

\newcommand{\coloneq}{\colonequals}

\equalenv{proposition}{prop}
\equalenv{corollary}{coro}
\equalenv{remark}{rema}
\equalenv{definition}{defi}
\equalenv{example}{exem}
\newtheorem*{problemone}{Problem 1}
\newtheorem*{problemtwo}{Problem 2}
\newtheorem*{example_nonumber}{Example}

\title
[Extending Whitney's extension theorem]
{Extending Whitney's extension theorem: \\ nonlinear function spaces}

\alttitle
{Extension du théorème d'extension de Whitney: espaces de fonctions non linéaires}

\author{\firstname{David} \middlename{Michael} \lastname{Roberts}}

\address{School of Mathematical Sciences\\
The University of Adelaide\\
North Terrace\\
Adelaide SA 5005 (Australia)}

\email{david.roberts@adelaide.edu}

\thanks{DMR is supported by the Australian Research Council's \emph{Discovery Projects} funding scheme (grant number DP180100383), funded by the Australian Government. AS was supported by the Einstein foundation Berlin.}

\author{\firstname{Alexander} \lastname{Schmeding}}
\address{Department of Mathematics\\
University of Bergen\\ 
P.O. Box 7803\\
5020 Bergen (Norway)}
\email{alexander.schmeding@uib.no}

  
\altkeywords{théorème de l’extension de Whitney, fonctions lisses sur des domaines fermés, jet de Whitney, cuspides polynomiales, espace de Fréchet, submersion, variétés à bord anguleux, variétés à bords non-lisses, variétés d'applications, loi de l’exponentielle}


\begin{document}
\theoremstyle{remark}
\newtheorem{setup}[cdrthm]{}
\begin{abstract}
We consider a global, nonlinear version of the Whitney extension problem for manifold-valued smooth functions on closed domains $C$, with non-smooth boundary, in possibly non-compact manifolds.
Assuming $C$ is a submanifold with corners, or is compact and locally convex with rough boundary, we prove that the restriction map from everywhere-defined functions is a submersion of locally convex manifolds and so admits local linear splittings on charts.
This is achieved by considering the corresponding restriction map for locally convex spaces of compactly-supported sections of vector bundles, allowing the even more general case where $C$ only has mild restrictions on inward and outward cusps, and proving the existence of an extension operator.
\end{abstract}

\begin{altabstract}
Nous considérons une version du problème de l’extension de Whitney, globale et non linéaire, pour les fonctions lisses à valeurs dans des variétés et définies sur des domaines fermés $C$ à bords non-lisses dans des variétés possiblement non compactes.
Supposant que $C$ est une sous-variété à bord anguleux, ou qu'elle est compacte et localement convexe à bords non-lisses, nous montrons que l’opérateur de restriction, à partir des fonctions définies partout, est une submersion de variétés localement convexes, et donc permet des scindages linéaires locaux sur les cartes.
Nous considérons à cet effet l’opérateur de restriction correspondant pour les espaces localement convexes de sections de fibrés vectoriels à support compact, permettant aussi le cas plus général où $C$ n’a que des restrictions légères sur les cusps vers l’intérieur et l’extérieur, et montrons l’existence d’un opérateur de prolongement.
\end{altabstract}

\maketitle

\section{Introduction}

The extension of differentiable real-valued functions from closed subsets of Euclidean space, was decisively solved in the $1$-dimensional case by Whitney \cite{whitney} and for finite order differentiability on $\mathbb{R}^n$ by Fefferman \cite{MR2233850}. 
In the $1$-dimensional case, for any closed set $C\subset \mathbb{R}$ and the data of a \emph{Whitney jet} on $C$---a formal Taylor series defined everywhere on $C$---there is a real-valued differentiable function on $\mathbb{R}$ whose Taylor series on $C$ coincide with the original Whitney jet (the \emph{recognition problem}). 
Moreover, the restriction map from globally-defined functions to jets on $C$ is not just surjective but has a continuous (linear) section, also called an extension operator (the \emph{operator problem}), solved in \cite{MR1501749} and \cite{MR2373373}, the latter for $C^m$ functions on $\mathbb{R}^n$. 
We shall refer to these two problems jointly as the \emph{Whitney extension problem}.

For closed sets in $n$-dimensional space the differentiability class of the functions starts to impact the results and techniques (contrast \cite{Frerick_07b} with \cite{MR2373373}, for example), as well as regularity assumptions on the boundary of $C$ in the smooth case \cite[Theorem~2.1]{Frerick_07b} (obstructions to the operator problem).
A variation of these problems (\cite[Proposition~2.16]{MR671464}, \cite{Frerick_07b}) is to consider not the data of Whitney jets on $C$, but continuous extensions of differentiable functions from the interior $C^\circ$ of $C$, all of whose derivatives also extend continuously to $C$.
In this case the roughness of the boundary can prevent such smooth functions from defining Whitney jets (obstructions to the recogition problem) and from extending to a larger domain.
It is in this more delicate setting that we prove our theorems. 
Namely, we consider the following generalisation of the Whitney extension problem: Let $M$ and $N$ be smooth manifolds and $C\subset M$ closed. Recall that there is a smooth manifold structure on the space $C^\infty(M,N)$ of smooth maps from $M$ to $N$, modelled on spaces of sections in certain vector bundles (cf.\ Appendix \ref{app: calculus}).

\begin{problemone}\label{problem:1}
  To what extent and under what conditions can one define extension operators for the restriction map
  \[
    \res_C\colon C^\infty(M,N) \to \{\text{smooth functions }C\to N\}?
  \]
 Here by `smooth function' we mean smooth on $C^\circ$ such that all derivatives extend continuously to $C$. Part of this problem is to determine the appropriate definition of, and structure on, the latter function space.
\end{problemone}

Before we give an answer to this problem in Theorem~B stated below, let us illustrate an example toy application.  

\begin{example_nonumber}
  Let the manifold $M$ be a torus, an infinite cylinder, or more generally a quotient $\mathbb{R}^n/\Gamma$ by a proper action of some discrete group $\Gamma$, and let $C$ be the closure of an open set $C^\circ \subset M$ with non-smooth boundary.
  Take $N$ to be a Lie group $G$, and let $f\colon C \to G$ be a given function. 
  Assuming $f$ extends to a smooth function $M\to G$, when can we find an extension operator whose domain is all smooth functions $C\to G$ sufficiently close to $f$?
\end{example_nonumber}

If $G$ were a connected solvable Lie group, hence an aspherical manifold, one could reduce this problem to one of finding an extension operator for functions $C' \to \widetilde{G} \simeq \mathbb{R}^k$, where $C'\subset \mathbb{R}^n$ is the preimage of $C$ under the covering map $\mathbb{R}^n \to M$.
We might then apply existing extension theorems (e.g.\ \cite{MR2300454}), but then one needs to guarantee the resulting extended functions $\mathbb{R}^n \to \widetilde{G} \to G$ are $\Gamma$-equivariant so as to descend to $M$ (e.g.\ by averaging over $\Gamma$, if finite).
This approach fails even under mild generalisation, for instance to a non-aspherical homogeneous space on the source or non-solvable Lie group in the target.

One result of this article (Theorem~\ref{thm: mm:rb}) is that under suitable assumptions on $C$ (independent of being a subset of $M$) we can define a smooth, locally convex manifold structure on the codomain of $\res_C$. 
More precisely, if $C$ is compact and a manifold with rough boundary (Definition~\ref{defn: RBM}) then the space of smooth functions $C\to N$ is a Fr\'echet manifold with charts modelled on space of sections of vector bundles over $C$.
We recall also that if $C$ is in fact a manifold with corners (a special case of having a rough boundary), then we can drop the assumption of compactness and recover the construction of Michor \cite[Theorem~10.4]{michor} of a smooth manifold of smooth maps $C\to N$.
In this case, the charts are given by \emph{compactly supported} spaces of sections, so it is in this generality we will work. \\[2.3mm]
\textbf{Remark}
  The restriction \emph{in the general rough boundary case} to compact sets $C$ is only due to current manifold of mappings technology; a generalised $\Omega$-lemma in the forthcoming \cite{GloecknerNeebBuch} is one main missing ingredient. 
  We conjecture that the results of this paper relying on compactness of $C$ will be true for non-compact $C$.
\\[2.3mm]
Given Theorem~\ref{thm: mm:rb}, then, the nonlinear map $\res_C$ looks like, on charts, a \emph{linear} restriction map for spaces of (compactly supported) sections of vector bundles.
We thus attack Problem~1 by reducing it to a linear extension problem involving spaces of sections of vector bundles.
To set this problem up, let $M$ be a smooth manifold, $E\to M$ be a real vector bundle and $C\subset M$ some closed set.
Let $\Gamma_c(M,E)$ denote the space of compactly-supported smooth sections of $E\to M$.

\begin{problemtwo}\label{problem:2}
  Under what conditions can one define extension operators for the restriction map
  \[
    \res_C\colon \Gamma_c(M,E) \to \{\text{compactly-supported smooth sections } C\to E\big|_C\}?
  \]
 Again, part of this problem is to determine the appropriate definition of and structure on the latter function space.
\end{problemtwo}

One can make much weaker assumptions on the closed set in this case, leading to a much stronger theorem than we need for the application to Problem~1.
An answer to this problem is given in Theorem~A below, but let us first consider a special case.

\begin{example_nonumber}
  Consider again a quotient manifold $M = \mathbb{R}^n/\Gamma$ as in the previous Example, and a complex line bundle $E\to M$. 
  Sections of $E$ can be identified with functions $\mathbb{R}^n \to \mathbb{C}$ satisfying a twisted equivariance conditions.
  For a closed set $C\subset M$ with non-smooth boundary, when is there an extension operator from smooth sections over $C$ to global smooth sections?
\end{example_nonumber}

Again, under very special assumptions on the geometry, existing results (e.g.\ \cite[Theorem~2.1]{MR2300454}) might be adapted as in Example~1, since the space of sections of a rank-$k$ vector bundle $E \to \mathbb{R}^n/\Gamma$ is isomorphic to a space of suitably twisted-equivariant functions $\mathbb{R}^n \to \mathbb{C}^k$ (for instance, using a family of $GL(k,\mathbb{C})$-valued multipliers); this approach fails under mild generalisation.

\subsection*{Statement of results}

We describe our results now in more detail, starting from the linear case (Theorem~A) and working up to the main, nonlinear case (Theorem~B). 
Fix a pair of finite-dimensional manifolds $M,N$ with $M$ being a $\sigma$-compact and equipped with a Riemannian metric.
Let $C \subset M$ be a closed set satisfying a cusp condition, defined below in Definition~\ref{def:cusp_condition}.
This condition allows general Lipschitz domains, but also much rougher boundary conditions, for instance Koch snowflake-like sets.
Note that at this stage we do not assume that $C$ carries any submanifold structure of its own, whence smoothness is only a meaningful concept because we can test in charts of the manifolds $M$ and $N$ which do not have a boundary.  

\paragraph{Theorem A} \emph{
Let $E\to M$ be a (finite-rank) vector bundle and $C\subset M$ 
a closed set satisfying the cusp condition.
The restriction map on compactly-supported smooth sections 
\begin{align*}
\res_C\colon \Gamma_c(M,E) & \to \Gamma_c(C,E)\\ 
\sigma & \mapsto \sigma|_C
\end{align*}  has a continuous linear splitting.
}\smallskip

We will use this result to show that $\res^M_C \colon C^\infty (M,N) \rightarrow C^\infty (C,N)$ admits \emph{local} splittings.
Recall that for a smooth manifold $M$ (possibly with corners), the space $C^\infty_{\text{fS}} (M,N)$ of smooth mappings with the fine very strong topology (see \cite{HS17} and \cite{michor}, where the fS-topology is called $\mathcal{FD}$-topology) can be turned into an infinite dimensional manifold. 
If $M$ is compact the fine very strong topology coincides with the well known compact open $C^\infty$-topology.
We prove in Section \ref{sect: MFDMAP} that Theorem~A yields local sections of $\res_C^M$ if $C$ is a submanifold with corners of $M$.
If $C$ is compact, we can even relax the condition and allow \emph{submanifolds with rough boundary}, a definition introduced by Karl-Hermann Neeb \cite{GloecknerNeebBuch}. 
Thus our next main result can be formulated as follows:

\paragraph{Theorem B}\emph{
For $C\subset M$ a submanifold with corners, or compact and a submanifold with rough boundary, the restriction map $\res_C^M \colon C^\infty_{\text{fS}} (M,N) \rightarrow C^\infty_{\text{fS}} (C,N)$ is a submersion of locally convex manifolds.
}\smallskip

Recall that for infinite-dimensional manifolds whose model spaces are more general than Banach spaces, a submersion is a map that locally, in submersion charts, looks like a projection out of a product.
This is a stronger condition than the map on tangent spaces being a split surjection (cf.\ \cite{1502.05795v4} for a detailed study).

We remark here that Theorem~B does not imply that $\res_C^M$ is surjective as not necessarily all smooth functions on closed submanifolds with (rough) boundary will admit extensions to the ambient manifold (compare \cite[Corollary 6.27]{MR2954043}).
A simple example is the case where $M=S^2$, $C\subset S^2$ is a closed equatorial `belt' and $N=S^1$. 
A map $C \to S^1$ cannot extend to $S^2$ if has non-zero winding number.

Finally, we look at nested closed subsets which satisfy the assumptions of Theorem~B. 

\paragraph{Corollary C} \emph{
With the manifolds $M,N$ as above and closed sets $C\subset D\subset M$ which both satisfy the assumptions of Theorem B, the restriction map 
\begin{align*}
\res_C^{D}\colon C^\infty_{\text{fS}} (D,N) & \to C^\infty_{\text{fS}} (C,N)\\ 
f & \mapsto f|_{C}
\end{align*}
is a submersion of locally convex manifolds.
}\smallskip

A more specific corollary applies the above collection of results to closed sets that are geodesically strongly convex, for example closures $\overline{U_{i\ldots j}}$ of iterated finite intersections  $U_{i\ldots j} = U_i \cap \ldots \cap U_j$ of geodesically strongly convex charts. 
Such closed sets satisfy the required cusp condition and we prove in Lemma \ref{lem: sconv:rbsmfd} that they are submanifolds with rough boundary.

\paragraph{Corollary D}\emph{
  Let $M$ be a smooth Riemannian manifold with geodesically 
  strongly convex compact sets $C \subset D \subset M$ and $N$ another smooth manifold.
  Then the restriction map}
  \[
    \res_C^{D}\colon C^\infty (D,N) 
    \to 
    C^\infty(C,N) 
  \]
\emph{  is a submersion of \Frechet manifolds.}\smallskip

Corollary~D allows the construction of various spaces of tuples of maps satisfying equations on suitable closed subsets of their domain; one can use the submersions it gives to ensure certain limits of diagrams of \Frechet manifolds exist. 
To this end, a close analogue of this corollary was stated as \cite[Proposition 3]{1610.05904}, with only a rough sketch of a proof, ignoring the function space topologies, and also allowing $M$ to be a manifold with corners.
This was used to construct infinite-dimensional manifolds of certain functors from a \v Cech groupoid to an arbitrary Lie groupoid.
However, the correct hypothesis is rather `rough boundary', rather than corners, so Corollary~D should be taken to replace \cite[Proposition 3]{1610.05904}. 

One can ask the obvious questions as to how much further the results here can be pushed, especially in light of the results of Frerick on general sets satisfying the cusp condition \cite{MR2300454}.
The biggest obstacle in pursuing this, is to define the relevant locally convex topologies or manifold structures in the linear and non-linear cases respectively. In light of this, an extension of the results in the present paper might be possible but there seems to be no straightforward way to do this.

\smallskip

A brief outline of the paper is as follows. 
In Section~\ref{sect:prelim} we give basic notions that are needed for the paper, relegating most technical results for infinite-dimensional calculus and manifolds to Appendix~\ref{app: calculus}. 
Section~\ref{sect: WElin} gives the necessary ingredients to build towards Theorem~A, namely various bits of extension theory and patching results in the linear setting, and these are assembled in Section~\ref{sect: proof of thm A}.
We then recall (from the forthcoming \cite[Chapter 1.4]{GloecknerNeebBuch}) the fundamentals of the theory of manifolds with rough boundary in Section~\ref{sect: mfdmap:rb} and construct the smooth manifolds of maps in that case. 
In Section~\ref{sect: MFDMAP} we then finally prove Theorem~B.
Appendix~\ref{app: Whitney} is a summary of the theory of Whitney jets, for ease of reference.

\subsection*{Acknowledgements}

The first author thanks various commenters on MathOverflow, especially Uri Bader, as well as Beno\^it Kloeckner for helpful discussions.
The second author thanks Helge Gl\"ockner for useful comments on manifolds with rough boundary and their spaces of mappings (in particular that Theorem \ref{thm: mm:rb} follows from the exponential law). Both authors thank Seppo Hiltunen who made them aware of a critical error in an earlier version of this paper, and the referee for their helpful comments.

\section{Preliminaries and Notation}\label{sect:prelim}

We wish to study an extension operator between spaces of smooth functions on manifolds. In the end, we will see that, as for the vector space case, an extension operator for functions defined on a ``suitably nice'' subset $C$ of a manifold $M$ to smooth functions on the whole manifold exists. Further, we want to establish that the \emph{restriction of $N$-valued functions} is a submersion in the sense of \cite{1502.05795v4}.

\begin{setup}[Notation and conventions]
 We write $\N \coloneq \{1,2,\ldots\}$ and $\N_0 \coloneq \N \cup \{0\}$.
 Frequently we will use standard multiindex notation to denote (iterated) partial derivatives of a (smooth) function $f \colon \R^d \supseteq U \rightarrow \R^m$ as $\partial^\alpha f$ for $\alpha \in \N_0^d$ (see \ref{defn: multiindex}).
 For a subset $S$ of a topological space we denote by $S^\circ$ its interior. 
 We say that a subset $C$ of a topological space is \emph{regular}, if $C^\circ$ is dense in $C$. 
 We note that closed subsets satisfying the cusp condition to be defined below are always regular.
 
 Further, every finite-dimensional manifold considered in the following will always be assumed to be Hausdorff and $\sigma$-compact. 
 
 We say $M$ is a \emph{Banach} (or \emph{Fr\'echet}) manifold if all its modelling spaces are Banach (or Fr\'echet) spaces. In general, infinite-dimensional manifolds will not required to be $\sigma$-compact or paracompact. 
\end{setup}
 
We consider functions on non-open sets following \cite{Wockel06} (where these mappings are used to define manifolds with boundary). Further we frequently have need for smooth functions on possibly infinite-dimensional manifolds (think manifold of mappings). 
To this end we base our investigation on the so called Bastiani calculus \cite{bastiani} which readily generalises beyond the realm of Banach spaces (cf.\ Appendix \ref{app: calculus} for a short introduction).

\begin{definition}
 Let $E,F $ be locally convex spaces and $C \subseteq E$ be a set with dense interior. A continuous mapping $f \colon C \rightarrow F$ is called $C^1$-map if 
 $f\big|_{C^\circ}$ is $C^1$ in the sense of Bastiani calculus and the derivative $d(f\big|_{C^\circ})$ extends (necessarily uniquely) to a continuous mapping $df \colon C \times E \rightarrow F$.
 
 Similarly we say $f$ is $C^k$ for $k \in \N \cup \{ \infty \}$ if $f\big|_{C^\circ}$ is $C^k$ and the iterated differentials extend (uniquely) to all of $C$. 
 We say $f$ is \emph{smooth (or $C^\infty$)} if $f$ is $C^k$ for every $k \in \N$ and write $C^\infty (C,F)$ for the set of all smooth maps on $C$.
\end{definition}

We have the following version of the chain rule (cf \cite[Remark 5]{Wockel06}, and \cite[Lemma 3.17]{alas2012} for a more general statement; this is also treated in \cite[Proposition 1.4.10]{GloecknerNeebBuch}):

\begin{lemma}\label{lem: reschain}
 Let $C \subseteq E$ and $D \subseteq F$ be regular subsets of locally convex spaces $E,F$ and $H$ be another locally convex space.
 Consider $C^k$ mappings $f \colon C \rightarrow D \subseteq F$ and $g \colon D \rightarrow H$ then $g \circ f$ is a $C^k$ mapping if one of the following conditions is satisfied
 \begin{enumerate}
 \item $f(C^\circ) \subseteq D^\circ$ (no condition on $C$ and $D$), 
 \item $C,D$ are locally convex sets, i.e.\ every point has a neighborhood in the set which is convex (no condition on $f$ and $g$).
 \end{enumerate}
\end{lemma}

Note that for an open set $C \cap U$ is a regular set if $C$ is regular.  
Thus the chain rule allows us to make sense of  $C^k$-mappings on regular subsets of smooth manifolds \textbf{without boundary}. 

\begin{definition}
Let $C \subseteq M$ be a regular subset of a manifold without boundary. A continuous map $f \colon C \rightarrow N$ to a manifold $N$ without boundary is a $C^k$-mapping if for every $x\in C$ there is a pair of charts $(\varphi, U), (\psi,V)$ with $x\in U$, $f(x)\in V$ such that $\psi \circ f \circ \varphi^{-1}|_{\varphi (U \cap C)}$ makes sense and is a $C^k$-mapping.
\end{definition}

Clearly by Lemma \ref{lem: reschain} condition 1.\ this definition is independent of the choice of charts. 
However, we note that many of the familiar rules of calculus are no longer valid for $C^k$-mappings on sets with dense interior which are not locally convex. 
In any case, these results are not needed to treat spaces of sections as locally convex spaces in Section \ref{sect: WElin} below and to prove Theorem A.

To retain the ``usual behaviour'' of differentiable functions (most importantly, the Mean Value Theorem, and hence the chain rule) it is well known (e.g. \cite{Keller}) one needs to work with locally convex topological vector spaces. In the non-linear setting one needs to require in addition that the subset $C$ of the domain manifold is locally convex\footnote{This observation seems to be due to Karl-Hermann Neeb, and will be treated in the forthcoming book \cite{GloecknerNeebBuch}.}. 
This will be important to establish the global setting required in Theorem B. 
Namely, the usual rules of calculus enable the construction of manifolds of mappings as outlined in Section \ref{sect: mfdmap:rb}.

\section{Whitney's extension theorem for linear spaces of functions}\label{sect: WElin}

The aim of this section is to recall the Whitney extension theorem in the vector space case. 
Further, we discuss conditions under which the space of Whitney jets can be identified with spaces of smooth functions on a regular closed set. 
In this section we let $C \subseteq \R^d, d \in \N$ be a regular closed set.

\begin{setup}[Ideals of functions vanishing on closed sets]
Let $m \in \N$ and $W \subseteq \R^d$ be an open neighborhood of the regular closed set $C$. We consider 
\[
\mathcal{I}_C (W,\R^m) \coloneq \{g \in C^\infty (W,\R^m) \mid \partial^\alpha g|_C \equiv 0 \quad \forall \alpha \in \N_0^d \}.
\]
Since $\partial^\alpha \colon C^\infty_{\text{co}} (U,\R^m) \rightarrow C_{\text{co}} (U,\R^m), f \mapsto \partial^\alpha f$ and $\ev_x \colon C_{\text{co}} (W,\R^m) \rightarrow \R^m, f \mapsto f(x)$ are continuous linear (cf.\ \cite[Definition 2.5 and Proposition 3.20]{alas2012} with respect to the compact open $C^\infty$-topology (cf.\ Appendix \ref{app: calculus}), 
\[
\mathcal{I}_C (W,\R^m) = \bigcap_{\alpha \in \N_0^d} \bigcap_{x \in C} (\ev_x \circ \partial^\alpha)^{-1} (0)
\] 
is a closed vector subspace of the \Frechet space $C^\infty_{co} (U,\R^m)$.
Indeed, if we denote by $\mathcal{E} (C,\R^m)$ the $\R^m$-valued Whitney jets on $C$ (see Appendix \ref{app: Whitney}), we can view $\mathcal{I}_C (W,\R^m)$ as the kernel of the linear map $r_W \colon C^\infty (W,\R^m) \rightarrow \mathcal{E}(C,\R^m), g \mapsto (\partial^\alpha g\big|_C)_\alpha$. 
Recall from \cite[p.\ 126]{MR2300454} that $r_W$ is continuous if $m=1$.\footnote{Indeed the article claims this only for $W=\R^d$ but continuity follows directly from the remarks above Definition 2.1 in loc.~cit. as explained in Remark \ref{Whitney:fun}.} 
Identifying $C^\infty_{co} (W,\R^m) \cong C^\infty_{co} (W,\R)^m$ (cf.\ \cite[Lemma 3.4]{MR1934608}) we obtain continuity of $r_W$ for arbitrary neighborhoods $W$ and $m \in \N$. 
\end{setup}

\begin{thm}[{Whitney extension theorem \cite[Theorem 1]{whitney}, or \cite[Theorem 2.2]{MR2300454} for a modern introduction}]\label{thm: whitney:ext}
The following sequence of \Frechet spaces is exact:
\begin{equation}\label{eq: seq0}
\begin{xy}
\xymatrix{
    0 \ar[r] & \mathcal{I}_C (W, \R^m) \ar[r] & C^\infty_{\text{co}} (W, \R^m) \ar[r] & \mathcal{E} (C, \R^m)\ar[r] & 0 
    }
\end{xy}
\end{equation}
\end{thm}

\begin{remark}
Recall that in the category of locally convex spaces, a sequence 
\begin{displaymath}
\begin{xy}
\xymatrix{
    0 \ar[r] & A \ar[r]^i & B \ar[r]^q & C \ar[r] & 0 
    }
\end{xy}
\end{displaymath}
of continuous linear maps is \emph{exact} if it satisfies both of the following conditions
\begin{enumerate}
\item \emph{algebraically exact}, i.e.\ images of maps coincide with kernels of the next map,
\item \emph{topologically exact}, i.e.\ $i$ and $q$ are open mappings onto their images. 
\end{enumerate}
If $A$, $B$ and $C$ are \Frechet spaces topological exactness follows from algebraic exactness by virtue of the open mapping theorem; for general locally convex spaces this is not the case (cf.\ e.g.\ \cite{MR1977923}).
\end{remark}

Note that the Whitney extension theorem in general requires only a closed set $C$ and not (as we required) a closed and regular set. 
However, in our approach we will replace the space of Whitney jets by a space of smooth functions on a closed set. 
Here the regularity assumption comes into play (cf.\ Appendix \ref{app: calculus}) and we will now construct a mapping which deals with the identification:  

\begin{setup}\label{Dmap}
Consider the mapping 
\begin{align*}
D \colon C^\infty (C,\R^m) &\rightarrow \prod_{\alpha \in \N_0^d} C_{\text{co}}(C,\R^m) \\ 
f & \mapsto (\partial^\alpha f)_{\alpha}.
\end{align*}
Then $D$ makes sense by our definition of $C^\infty (C,\R^m)$ and is injective and linear. 
Arguing as in \cite[Section 2]{MR2300454} the image of $D$ is a closed subspace of the \Frechet space $\prod_{\alpha \in \N_0^d} C_{\text{co}}(C,\R^m)$ (note that we have compact convergence of functions and all derivatives on the dense interior of $C$!). 
\end{setup}

As the mapping $D$ takes a smooth function on $C$ to a jet expansion (i.e.\ its family of derivatives), one is tempted to think that $D$ takes its image in the space $\mathcal{E}(C,\R^m)$ of Whitney jets. 
However, this is wrong in general as the following example from \cite[Example 2.18]{MR671464} shows:

\begin{example}
Let $C$ be the complement of the open subset $\{(x,y) \in \R^2 \mid 0 < y < \exp (-1/x^2), x > 0\}$. 
Then $C$ is a regular closed set and we define a function $f \in C^\infty (C,\R^2)$ as follows 
\[
f(x,y) = \begin{cases} \exp (-1/x^2) & \text{ if } x>0 ,\ y \geq \exp (-1/x^2) \\
0 & \text{ otherwise }
\end{cases}
\]
A computation of $\left(f\!\left(x,\exp(-1/x^2)\right)-f(x,0)\right)/(\exp(-1/x^2)-0) =1$ shows that $f$ cannot be extended to a smooth function on $\R^2$. 
Thus in particular, the image of $f$ under the mapping $D$ from \ref{Dmap} is not a Whitney jet. 
\end{example}

As a consequence $D$ can take its image in the space of Whitney jets only if every smooth function on $C$ extends to a smooth function on an open neighborhood of $C$. 
It turns out that the non existence of extensions is tied to the exponential type cusps of the set $X$ in the example. 
Prohibiting such inward cusps, which we shall call \emph{narrow fjords}, ensures that every smooth function can indeed be extended.

\begin{definition}[{\cite[2.16.1]{MR671464}}]\label{defn: incusp}
Let $A$ be a regular closed subset of $\R^d$. We say $A$ has \emph{no narrow fjords} if for all $a\in A$ exists an integer $p$, a compact neighborhood $K$ of $a$ in $A$ and a constant $C>0$ such that any $x,y\in K$ can be joined by a rectifiable path $\gamma$ lying inside $A^\circ$, except perhaps for finitely many points, and the length $\ell (\gamma)$ of $\gamma$ satisfies
\[
  \lVert x-y\rVert \geq C\,\ell (\gamma)^p.
\]
\end{definition}

Note that this definition gives control over how fast the width of fjords can shrink as one moves inwards along them, see Figure~\ref{fig:fjords_condition}. Further, the no narrow fjords condition is closely related to the conditions called $C$-quasiconvexity and the $(C,\omega)$-convexity from \cite[Definition 2.63]{MR2882877}.

\begin{figure}
\begin{center}
\usetikzlibrary{decorations.markings}
\begin{tikzpicture}[decoration={ 
markings,
mark=at position 0.08 with {\arrow{stealth}};,
mark=at position 0.3 with {\arrow{stealth}};,
mark=at position 0.8 with {\arrow{stealth}}
}]

  \path [fill=lightgray] 
        (-1,0)
        to (-1,5.5)
        to (6,5.5) 
        to (5.5,4.5)
        to (4,4.5)
        to [out=-50] (4.5,3.5)
        to [out=-165,in=0] (0.5,3)
        to [out=0,in=170] (3.5,2.5)
        to (4,2.75) 
        to [out=-80,in=135] (5,1.5)
        to [out=-75] (6,0)
        to (-1,0);
  \draw (6,5.5) to (5.5,4.5)
        to (4,4.5)
        to [out=-50] (4.5,3.5)
        to [out=-165,in=0] (0.5,3)
        to [out=0,in=170] (3.5,2.5)
        to (4,2.75) 
        to [out=-80,in=135] (5,1.5)
        to [out=-75] (6,0);
  \node [below] at (0,0.5) {$A$};

  \draw [fill] (0.5,3) circle [radius=0.03];
  \node [left] at (0.5,3) {$a$};

  \draw (5.19,1) -- (-0.5,1) -- (-0.5,4.63)  -- (5.56,4.63);
  \node [above right] at (-0.5,4.15) {$K$};

  \draw [fill] (5.5,4.5) circle [radius=0.03];
  \node [below left] at (5.5,4.5) {$x$};

  \draw [fill] (5,1.5) circle [radius=0.03];
  \node [right] at (5,1.5) {$y$};

  \draw [dotted] (5.5,4.5) to (5,1.5);

  \draw[postaction={decorate}, red] (5.5,4.5) to (5,4.75)
              to (4,4.5)
              to (0.5,3)
              to (5,1.5);
  \node [left] at (3,4.2) {$\gamma$};

  \draw [fill] (3.5,3.25) circle [radius=0.03];

  \draw [fill] (3.5,2.5) circle [radius=0.03];

  \draw [dotted] (3.5,3.25) to (3.5,2.5);


  \draw [fill] (2.3,3.07) circle [radius=0.03];

  \draw [fill] (2.3,2.76) circle [radius=0.03];

  \draw [dotted] (2.3,3.07) to (2.3,2.76);

\end{tikzpicture}

\caption{No narrow fjords condition}\label{fig:fjords_condition}
\end{center}
\end{figure}
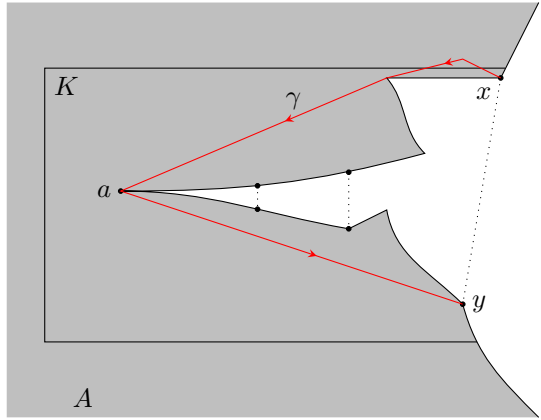

\begin{example}
  Let $A$ be a regular closed set.
  Recall that the open set $A^\circ$ satisfes the \emph{bounded turning condition} if there is a constant $C>0$ such that for all $x,y\in A$, there is a rectifiable path $\gamma$ from $x$ to $y$ such that $\lVert x - y\rVert \geq C\, \ell(\gamma)$. 
  If $A^\circ$ satisfies the bounded turning condition then $A$ has no narrow fjords.
  Any uniform domain \cite{Martio-Sarvas} (see, for example, \cite[Definition 2.2]{Mourgoglou} for an updated formulation) satisfies the bounded turning condition, which includes all H\"older domains and NTA domains ("non-tangentially accessible domains" as introduced by \cite{JK82}), and so the closures of all these sets all have no narrow fjords.
\end{example}

\begin{setup}\label{setup: WJ}
Let now $C$ be a regular closed set with no narrow fjords.
Then $D \colon C^\infty (C,\R^m) \rightarrow \prod_{\alpha \in \N_0^d} C_{\text{co}} (C,\R^m)$
takes its image in $\mathcal{E}(C,\R^m)$ by  \cite[Proposition 2.16]{MR671464}.
As a consequence of the Whitney extension theorem \ref{thm: whitney:ext}, every element in $C^{\infty} (C,\R^m)$ extends to a smooth map on $\R^d$, whence the image of $D$ coincides with the space of Whitney jets $\mathcal{E} (C, \R^m)$ (see Definition \ref{defn: WJm}).
Thus we topologize $C^\infty (C,\R^m)$ with the identification topology induced by $D$, turning it into a \Frechet space isomorphic to the space of Whitney jets on $C$.
In particular, the exact sequence \eqref{eq: seq0} yields an exact sequence of \Frechet spaces
\begin{equation}\label{eq: WSQ1}
  \begin{xy}
    \xymatrix{
        0 \ar[r] & \mathcal{I}_C (W, \R^m) \ar[r] & C^\infty_{\text{co}} (W, \R^m) \ar[r] & C^\infty  (C, \R^m)\ar[r] & 0 .
        }
  \end{xy}
\end{equation}
\end{setup}

In the next section we are going to investigate \emph{outward} cusp conditions on the boundary of closed subsets and show how they can be transferred to Riemannian manifolds.

\section{The cusp condition}

In the last section we have already encountered a cusp condition preventing the occurrence of certain (inward) cusps on the boundary of the closed set on which we are working.
The key functional-analytic result we use to extend sections is due to Frerick in \cite{MR2300454}.
It uses a metric condition on a closed domain $F$ in $\R^n$ to ensure there is a continuous extension operator for Whitney jets on $F$ to smooth functions on $\R^n$.
The following definition abstracts the hypothesis from \cite[Theorem 3.16]{MR2300454} and from Definition \ref{defn: incusp} so as to apply to closed sets in a metric space more general than $\R^n$. See Figure~\ref{fig:cusp_condition} for an illustration of the various quantities in the following definition.

\begin{definition}\label{def:cusp_condition}
  Let $(M,d)$ be a locally compact metric space. A closed set $F\subset M$ has 
  \begin{enumerate}
  \item \emph{no narrow fjords} if for each $x \in F$ there exists $p\in \N$, $K \subseteq F$ a compact neighborhood of $x$ and $D>0$ such that all $y,z \in K$ can be joined by a rectifiable curve $\gamma$ lying inside $F^\circ$, except perhaps for finitely many points, such that its length $\ell_d(\gamma)$ satisfies $d(y,z) \geq D\,\ell_d (\gamma)^p$.
  \item \emph{at worst polynomial outward cusps} if for all compact $K\subset M$ there exist $\varep_0,\rho >0$ and $r\geq 1$ such that for all $z\in K\cap \partial F$ and $0<\varep<\varep_0$ there is an $x\in F$ with $d(x,z)<\varep$ such that if $d(x,y)<\rho\varep^r$ then $y\in F$ and $d(z,y)<\varep$.
  \end{enumerate}
  If $F$ has at worst polynomial outward cusps and no narrow fjords we simply say that $F$ \emph{satisfies the cusp condition}. 
\end{definition}

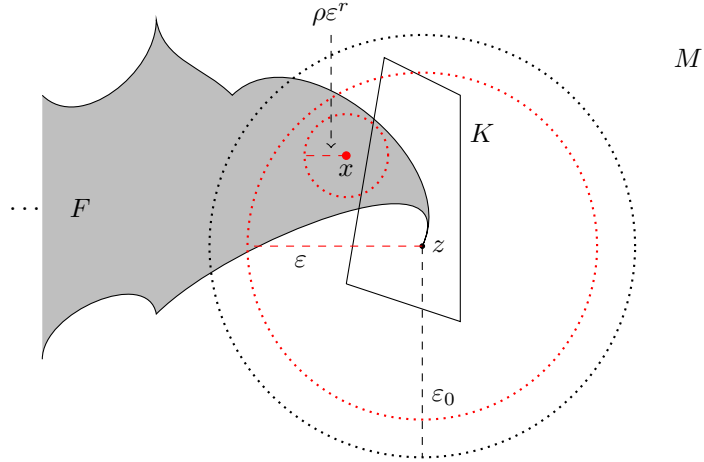
\begin{figure} 
\begin{center}

\begin{tikzpicture}

  \node at (10,6) {$M$};

  \draw [fill=lightgray] 
          (1.5,5.5) to [out=-45,in=-90] (3,6.5) 
          to [out=-80,in=135] (4,5.5) 
          to [out=45,in=64]  (6.5,3.5) 
          to [out=66,in=45] (3,2.6) 
          to [out=100,in=90] (1.5,2);
  \node at (2,4) {$F$};
  \node [left] at (1.6,4) {$\ldots$};

  \draw (5.5,3) -- (6,6) -- (7,5.5) -- (7,2.5) -- (5.5,3);
  \node [right] at (7,5) {$K$};

  \draw [fill] (6.5,3.5) circle [radius=0.03];
  \node [right] at (6.5,3.5) {$z$};

  \draw [thick,dotted] (6.5,3.5) circle [radius=2.8];
  \draw [dashed] (6.5,3.5) -- (6.5,0.7);
  \node [right] at (6.5,1.5) {$\varepsilon_0$};

  \draw [red, dotted, thick] (6.5,3.5) circle [radius=2.3];
  \draw [red,dashed] (6.5,3.5) -- (4.2,3.5);
  \node [below] at (4.9,3.5) {$\varepsilon$};
  
  \draw [fill=red,red] (5.5,4.7) circle [radius=0.05];
  \node [below] at (5.5,4.7) {$x$};

  \draw [red, dotted,thick] (5.5,4.7) circle [radius=0.55];
  \draw [red,dashed] (5.5,4.7) -- (4.95,4.7);
  \draw [->,dashed] (5.3,6.3) -- (5.3,4.8);
  \node [above] at (5.3,6.3) {$\rho\varepsilon^r$};

\end{tikzpicture}
\caption{Polynomial outward cusps}\label{fig:cusp_condition}
\end{center}
\end{figure}

In the case that $r=1$, the condition on outward polynomial cusps is sometimes called the (interior) corkscrew condition \cite[p.~123]{MR2305115}, and so our polynomial cusps can be seen as corkscrews with nonlinear growth. 

\begin{remark}\label{wlog_constants_nice}
  In Definition~\ref{def:cusp_condition}.2, if the constants $\varep_0,\rho,r$ work for the compact set $K$, then so do smaller such constants, and if $\varep_0 \leq 1$ then we can also increase $r$.
  Putting this together, we can assume that $\varep_0 = \rho < 1$ and increase $r$ as needed, and as a result can replace $\rho \varep^r$ by $\varep^{r+1}$.
  Hence we can, without loss of generality, assume that $\rho=1$ and $r \geq 2$.
\end{remark}

\begin{example}\label{ex: cusp}
  Every Lipschitz domain satisfies the cusp condition, as do H\"older domains and NTA domains. 
  The compact subset of $\R^2$ whose boundary is the Koch curve satisfies the cusp condition.
\end{example}

Recall that $\mathcal{E}(F)$ denotes the space of Whitney jets on the closed set $F$.
In the following Theorem, $\R^n$ is taken with the Euclidean metric.

\begin{theorem}[{\cite[Theorem 3.16]{MR2300454}}]\label{thm:Frerick}
  Let $F\subset \R^n$ be closed and have at worst polynomial outward cusps. 
  Then the surjective map $C^\infty(\R^n) \to \mathcal{E}(F)$ of Fr\'echet spaces has a continuous linear splitting. 
\end{theorem}

Moreover, Theorem \ref{thm:Frerick} combined with \ref{setup: WJ} yields the following Corollary which generalises \cite[Theorem~2.1]{Frerick_07b}.

\begin{corollary}\label{cor:cusps!}
Let $F\subset \R^n$ be closed and satisfy the cusp condition, then the surjective map $C^\infty(\R^n) \to C^\infty (F,\R)$ of Fr\'echet spaces has a continuous linear splitting.
\end{corollary}

We want to be able to sensibly transfer both Frerick's Theorem and Corollary \ref{cor:cusps!} in Euclidean space to a Riemannian manifold, so we will need a result that allows change of metric.
The following result is stated in more generality than we need, since it should be of independent interest.

\begin{lemma}\label{lem:transfer_cusp_condition}
  Let $(M,d_1)$ be a locally compact, complete metric space, $F\subset M$ be closed and let $F$ have at worst polynomial outward cusps using the metric $d_1$.
  If $d_2$ is another metric on $M$ that is locally bi-H\"older equivalent to $d_1$, then $F$ has at worst polynomial outward cusps using the metric $d_2$.
\end{lemma}

\begin{proof}
  Let $K\subset M$ be any compact set and $\varep_{0,1}$, $\rho_1$ and $r_1$ be the constants guaranteed to exist for $K$ by virtue of $F$ satisfying Definition~\ref{def:cusp_condition}.2 for $d_1$. 
  By Remark~\ref{wlog_constants_nice} we will assume $\rho_1 = 1$, $r_1\geq 2$ and $\varep_{0,1} < 1$.

  Define the compact set
  \[
    N \coloneq \overline{\{ x\in M \mid 
          d_1(F\cap K,x) < 2
          \text{ and } d_2(F\cap K,x) < 2\}}.
  \]
  Now as $d_1$ and $d_2$ are locally bi-H\"older equivalent there are constants $C \geq 1$ and $0<\alpha\leq 1$ such that 
  \[
    \frac{1}{C}\, d_1(a,b)^{\frac1\alpha} \leq d_2(a,b) \leq C\, d_1(a,b)^\alpha
  \]
  for all $a,b\in N$.
  
  Take 
  \begin{align*}
    \varep_{0,2} &\coloneq \min\{C\,\varep_{0,1}^\alpha,{\textstyle \frac12} \},\\
     \rho    &\coloneq \frac{1}{C^{1+r_1/\alpha^2}},\text{ and}\\
     r_2       &\geq r_1/\alpha^2\quad \text{such that } \rho\,\varep_{0,2}^{r_2} \leq \varep_{0,1}
  \end{align*}
  to be the putative uniform constants required so that $F$ satisfies Definition~\ref{def:cusp_condition}.2 for the metric $d_2$.
  Note that since $\varep_{0,2} < 1$ it does makes sense to enlarge $r_2$ until the upper bound on $\rho\,\varep_{0,2}^{r_2}$ is satisfied.

  Let $z\in \partial F\cap K$ be arbitrary, and take any $\varep_2$ such that $0<\varep_2 < \varep_{0,2}$.
  Define $\varep_1 = (\varep_2/C)^{\frac1\alpha}$. 
  Since $\varep_1  = (\varep_2/C)^{\frac1\alpha} < (\varep_{0,2}/C)^{\frac1\alpha} = \varep_{0,1}$ then there is an $x\in F$ satisfying $d_1(x,z) < \varep_1$ such that
  \[
    d_1(x,y) < \varep_1^r \quad \Rightarrow \quad d_1(z,y) < \varep_1\,\text{ and }\, y\in F.
  \]
  Note that as $z\in K$ and $d_1(x,z) < \varep_1 < \varep_{0,1} < 2$, we have $x\in N$.
  Hence $d_2(x,z) \leq C\, d_1(x,z)^\alpha < \, \varep_1^\alpha = \varep_2$, as required.

  Now take $y\in M$ such that $d_2(x,y) < \rho\, \varep_2^{r_2}$.
  Then $d_2(y,z) \leq d_2(y,x) + d_2(x,z) < \rho\,\varep_2^{r_2} + \varep_2 < \varep_{0,1} + 1 < 2$, and so $y\in N$. So we can calculate that
  \begin{align*}
    d_1(x,y) & \leq \big( C\, d_2(x,y)\big)^\alpha\\
         & < \left(C\, \rho_2\right)^\alpha \varep_2^{r_2 \alpha} \\
         & = \left(\frac{\varep_2^{\alpha^2r_2/r_1}}{C} \right)^{\frac{r_1}{\alpha}}\\
         & \leq \left(\frac{\varep_2}{C}\right)^{\frac{r_1}{\alpha}} = \varep_1^{r_1}
  \end{align*}
  where we have used that $\alpha^2r_2 \geq r_1$ and $\varep_2<1$.
  Using the cusp condition for $K$ in $d_1$,
  \begin{align*}
    d_1(z,y) & < \varep_1 \quad (\text{and } y\in F) \\
    \Rightarrow 
    d_2(z,y) & \leq C\, d_1(z,y)^\alpha \\
         & < C\,\varep_1^\alpha = \varep_2.
  \end{align*}
  Hence $F$ has at worst polynomial cusps for $d_2$.
\end{proof}

Note that if we have \emph{uniformly} bi-H\"older equivalent metrics then we can dispense with the assumption of completeness; the proof goes through the same without the need to define the compact set $N$.

We also have the following simple result for transferring the other half of the cusp condition.

\begin{lemma}\label{lem: trans:nnf}
  Let $(M,d_1)$ be a locally compact, complete metric space, $F\subset M$ be closed and let $F$ have no narrow fjords using the metric $d_1$. Then if $d_2$ is another metric on $M$ that is locally bi-Lipschitz to $d_1$, then $F$ has no narrow fjords using the metric $d_2$.
\end{lemma}

This follows once recalling that rectifiable paths can be taken to be Lipschitz functions $I\to M$.

\begin{corollary}
  Take a manifold $M$ with a continuous Riemannian metric $g$, and a locally bi-Lipschitz chart, $\phi \colon U \xrightarrow{\sim} \R^n$ on $M$. 
  Here $\R^n$ is given the Euclidean metric, and $U$ the restriction of the geodesic metric $d^g$ on $(M,g)$.
  If $C \subset M$ is closed and satisfies the cusp condition for the metric $d^g$, then $F = \phi(C\cap U)\subset \R^n$ satisfies the cusp condition in the Euclidean metric.
\end{corollary}

\begin{remark}\label{rem: standard:arguments}
  Note that by standard arguments\footnote{see eg. the answer by Beno\^it Kloeckner at \url{https://mathoverflow.net/a/236851/}}, every $C^1$ manifold with a continuous Riemannian metric $g$ has an atlas of charts that are locally bi-Lipschitz to Euclidean space, hence \emph{a fortiori} locally bi-H\"older.
\end{remark}

We can apply this (perhaps overly general) result to our setup, namely where we take a relatively compact smooth chart $U$ on the smooth manifold $M$. Observe that $C\cap U$ satisfies the cusp condition if $C$ satisfies it. Thus we obtain a regular and closed (in $U$!) subset which satisfies the no-narrow fjord condition, hence $C^\infty(C\cap U)$ is a \Frechet space with the topology from \ref{setup: WJ}.We have a commutative diagram of \Frechet spaces (cf.\ \ref{setup: WJ} and Appendix \ref{app: Whitney} for a description of the topologies)
\[
  \xymatrix{
    C^\infty_{\text{co}}(U)  \ar[d] & \ar[l]_-{\simeq} C^\infty_{\text{co}}(\R^n) \ar[d] \\
    C^\infty(C\cap U) &  \ar[l]_-{\simeq} C^\infty(F) \ar@/_1pc/[u]_{\text{Theorem \ref{thm:Frerick}}}
  }
\]
where the vertical arrows are surjective, and a continuous section of the restriction map $C^\infty(\R^n) \to C^\infty(F)$. Thus:
\begin{lemma}\label{lemma:splitting restriction in charts}
  Let $C \subset M$ be a 
  closed set satisfying the cusp condition and $U \xrightarrow{\simeq} \R^n$ be a smooth chart on $M$.
  Then the restriction map $C^\infty_{\text{co}}(U) \to C^\infty(C\cap U)$ of \Frechet spaces has a continuous section.
\end{lemma}

\section{Proof of Theorem A}\label{sect: proof of thm A}

In this section we provide the necessary details for the proof of Theorem A from the introduction.
As a first step, we consider spaces of sections on certain regular closed subsets of a Riemannian manifold. After these sections have been discussed, it will turn out that we only need to collect the bits and pieces from the previous sections to obtain the result.
Throughout this section we fix the following data:

\begin{setup}\label{setup: prep:ThmA}
From now on $M$ will be a $d$-dimensional $\sigma$-compact manifold with a fixed choice of Riemannian metric $g$, $E\rightarrow M$ a rank $m$-vector bundle and $C \subseteq M$ a 
 closed subset which satisfies the cusp condition with respect to the geodesic metric $d^g$.
 
 Further, we choose and fix auxiliary data as outlined in \ref{setup: PData}. In particular, denote the locally finite atlas by $(U_i, \varphi_i)_{i \in I}$, $C_i = U_i \cap C$ and the relatively compact charts by $V_i \subseteq U_i$ which satisfy $C \subseteq \bigcup_{i\in I} V_i$ . 
\end{setup}

The main idea of the proof of Theorem A is as follows: We take a section and use local triviality of the bundle to cut it into pieces which can be extended due to the cusp condition. Then we reassemble the pieces into a section by using a classical local to global approach with a partition of unity. 
In the next subsections we provide the necessary tools: First we define the spaces of sections, then we prepare the local to global result.

\subsection*{Smooth bundle sections on a closed set without narrow fjords}
\addcontentsline{toc}{subsection}{Smooth bundle sections on a regular closed set without narrow fjords}

Our first task is to construct a suitable topology for the vector space of sections into $E$ on $C$.

\begin{definition}\label{defn: sectsp:cl}
 For a regular closed set $C$ which has no narrow fjords we define \begin{align*}
 \Gamma_c (C,E) \coloneq \{\sigma \in C^\infty (C,E) \mid \pi_E \circ \sigma = \id_C \text{ and } \supp \sigma \text{ is compact}\}
 \end{align*}
 the \emph{compactly supported smooth sections on $C$}.
 Further, define 
 \begin{align*}
 \mathcal{I}_{c} (C,E) \coloneq \{\sigma \in \Gamma_c (M,E) \mid T^k_x \sigma = 0\quad \forall x \in C, k \in \N\}
 \end{align*}
 the subspace of all compactly supported sections vanishing (with all their derivatives) on $C$.\footnote{Here we use the notation $T^k = T \circ T \cdots \circ T (k$ times) to denote the $k$-fold iterated tangent functor $T$. Note that unpacking the definition of the iterated tangent functors (see e.g.\ \cite[Lemma 1.14]{MR1911979} for a local version) the vanishing of all iterated tangent functors at a point is equivalent to the vanishing of all iterated partial derivatives in any chart containing the point.}
\end{definition}

\begin{remark}\label{rem: closedideal}
 Clearly the pointwise operations turn $ \Gamma_c (C,E)$ and $\mathcal{I}_{c} (C,E)$ into vector spaces. One can argue as in the vector space case to see that $\mathcal{I}_{c} (C,E)$ is a closed subspace of $\Gamma_c (M,E)$ (with the fine very strong topology). Indeed using an atlas of $M$ we can use Lemma \ref{lem: top:sect} and Remark \ref{rem: loc:sect} to rewrite the problem in charts, where closedness follows from the argument in \ref{thm: whitney:ext}. (Avoiding localisation in charts, one can alternatively use Lemma \ref{lem: top:sect} together with \cite[Lemma 3.8]{MR1934608}.)
\end{remark}

\begin{setup}\label{setup: PData}
 Let $\cU = (U_i, \varphi_i)_{i \in I}$ be a locally finite atlas of relatively compact charts of $M$ such that $\varphi_i (U_i) = \R^{d}$ and there is a collection $\cV$ of open sets $V_i \subseteq \overline{V}_i \subseteq U_i, \ i \in I$ with  
 \begin{itemize}
  \item $C \subseteq \bigcup_{i \in I} V_i$
  \item $(\chi_i)_{i\in I}$ is a smooth partition of unity with $\supp \chi_i \subseteq V_i$
 \end{itemize}
We set $C_i \coloneq C \cap U_i$ for $i \in I$ and note that $\varphi_i (C_i) \subseteq \varphi_i (U_i) = \R^{d}$ is closed.
\end{setup}

The following proposition is (apart from the topological assertions and the fact that we are working with smooth functions and not jets) a folklore fact which easily follows from the Whitney extension theorem \ref{thm: whitney:ext} in charts and a gluing argument. Since this argument will be the basis of our construction we give full details.
 
\begin{proposition}[Whitney extension theorem for sections on a manifold]\label{prop: Whit:mfd}
The linear restriction map $\res_C \colon \Gamma_c (M,E) \rightarrow \Gamma_c (C,E)$ is surjective and endows $\Gamma_c (C,E)$ with a quotient topology such that 
\begin{equation}\label{eq: ex:vssect} 
\xymatrix@C1.9em{
0 \ar[r] & \mathcal{I}_c (C,E) \ar[r] \ar[d] & \Gamma_c (M,E) \ar[r]^{\res_C} \ar[d]^{\rho_\cU} & \Gamma_c (C,E) \ar[r] \ar[d]^r & 0 \\
 0 \ar[r] &   \bigoplus_i \mathcal{I}_{C_i} (U_i, \R^m) \ar[r] &  \bigoplus_{i} C^\infty_{\text{co}} (U_i, \R^m) \ar[r]^q & \bigoplus_{i} C^\infty (C_i, \R^m) \ar[r] & 0}
\end{equation}
is commutative with exact rows in the category of locally convex spaces.\\
Here $r\colon \Gamma_c (C,E) \rightarrow \bigoplus_{i \in I} C^\infty (C_i,\R^m)$ sends  $f \mapsto (\pr_2 \circ T\varphi_i \circ f|_{C_i})_{i\in I}$ and the spaces $C^\infty (C_i,\R^m)$ are topologised as in \ref{setup: WJ}.
\end{proposition}

\begin{proof}
Let us first deal with the lower row: Since $\varphi_i$ is a diffeomorphism, we can use precomposition by $\varphi_i$ to identify $C^\infty_{\text{co}} (U_i,\R^m) \cong C^\infty_{\text{co}} (\varphi_i(U_i),\R^m)$ and $C^\infty (C_i,\R^m) \cong C^\infty (\varphi_i(C_i),\R^m)$. Now $F_i :=\varphi_i (C_i)$ is a closed subset of the ambient space and $\varphi_i (U_i)$ is an open neighborhood of $F_i$. Moreover, since $C$ has no narrow fjords, Lemma \ref{lem: trans:nnf} implies that $F_i$ has no narrow fjords, whence \ref{setup: WJ} yields for every $i\in I$ an exact sequence 
$$\begin{xy}\xymatrix{
 0 \ar[r] &  \mathcal{I}_{C_i} (U_i, \R^m) \ar[r] &   C^\infty_{\text{co}} (U_i, \R^m) \ar[r]^q & C^\infty (C_i, \R^m) \ar[r] & 0}
\end{xy},$$
where we set $\mathcal{I}_{C_i} (U_i,\R^m) \cong \mathcal{I}_{F_i} (\varphi_i (U_i),\R^m)$, $C^\infty (C_i,\R^m) \cong C^\infty (F_i,\R^m)$ and suppress the identifications in the notation. Using that taking countable direct sums in the category of locally convex spaces is exact, we see that the lower row of \eqref{eq: ex:vssect} is exact.

By Lemma \ref{lem: top:sect} we have canonical embeddings $\rho_\cU$ of $\Gamma_c (M,E)$ into $\bigoplus_i \Gamma(E|_{U_i})$ and $\rho_\cV$ of $\Gamma_c (M,E)$ into $\bigoplus_i \Gamma (E|_{V_i})$. We identify $\Gamma (E|_{U_i}) \cong C^\infty_{\text{co}} (U_i ,\R^m)$ as in Remark \ref{rem: loc:sect} and suppress this in the notation. 
Since $\mathcal{I}_c (C, E)$ is clearly contained in the kernel of $\res_C$, we obtain a commutative diagram of vector spaces:
\begin{equation}\label{diag: res} 
 \xymatrix@C1.9em{
  0 \ar[r] &  \mathcal{I}_{c} (C, E) \ar@{-->}[d]^{\rho_\cU|_C} \ar[r] & \Gamma_c (M,E) \ar[r]^{\res_C} \ar[d]^{\rho_\cU} &\Gamma_c (C,E) \ar[d]^{r= (r_{C_i})_{i\in I}} & \\
  0 \ar[r] &   \bigoplus_i \mathcal{I}_{C_i} (U_i, \R^m) \ar[r] &  \bigoplus_{i} C^\infty_{\text{co}} (U_i, \R^m) \ar[r]^q & \bigoplus_{i} C^\infty (C_i, \R^m) \ar[r] & 0
  }
\end{equation} 
Here $r_{C_i} (f) \coloneq f|_{C_i},\ i\in I$ and $\rho_\cU|_C$ is induced from $\rho_\cU$ and realises $\mathcal{I}_{c} (C,E)$ as the closed subspace $\{(c_i)_{i\in I} \in \bigoplus_i \mathcal{I}_{C_i} (U_i, \R^m) \mid  c_i|_{U_i \cap U_j} = \Phi_{ij}(\id_M,c_j)|_{U_i\cap U_j}\}$.\footnote{Recall that point evaluations and postcomposition with fixed smooth functions are continuous in the compact open $C^\infty$-topology (see e.g.\ \cite{alas2012}). An easy adaption of the argument in \cite[proof of Lemma 3.21 (b)]{MR3328452} shows that the subspace indeed is a closed subspace of the direct sum.} Note that apart from the space $\Gamma_c (C,E)$ which is not yet topologised, \eqref{diag: res} is a commutative diagram in the category of locally convex spaces 

\paragraph{$\res_C$ is surjective.} Consider $f\in \Gamma_c (C,E)$ and choose a family $(g_i)_{i\in I} \in \bigoplus_{i\in I} C^\infty (U_i,\R^m)$ with $q((g_i)_i) = r (f)$. In general $(g_i)_{i\in I}$ will not be contained in the image of $\rho_\cU$, but we see that 
\begin{equation}\label{eq: cuts:equiv}
g_i|_{C_i\cap C_j} = f|_{C_i \cap C_j} = g_j|_{C_i \cap C_j} \quad \text{for every  }i,j\in I.
\end{equation}
Using the partition of unity from \ref{setup: PData}, we construct smooth functions on $M$ by continuing $\chi_j \cdot g_j|_{M\setminus U_j} \equiv 0$. Hence 
$$h_i \coloneq \sum_{j \in I} (\chi_j \cdot g_j)|_{V_i} \in C^\infty (V_i, \R^m),\quad i\in I,$$ 
By construction $(h_i)_{i\in I} \in \bigoplus_{i \in I} C^\infty (V_i,\R^m)$ and $h_i|_{V_i \cap V_j} = h_j|_{V_i\cap V_j}$ holds for every pair $(i,j)\in I^2$.
Thus $(h_i)_{i\in I}$ is contained in the image of $\rho_\mathcal{V}$ and we can choose $h \in \Gamma_c (M,E)$ with $\rho_V (h) = (h_i)_{i\in I}$.
Now \eqref{eq: cuts:equiv} implies that $h_i|_{V_i \cap C_j} =f|_{V_i \cap C_j}$. As the $V_i$ cover $C$ (see \ref{setup: PData}), we see that $\res_C (h)=f$. 
Thus $\res_C$ is surjective and we can endow $\Gamma_c (C,E)$ with the quotient topology, thus turning it into a locally convex space.

\paragraph{$r$ is continuous with respect to the quotient topology}
Follows directly from the commutativity of \eqref{diag: res} and the definition of the quotient topology. Note that $r$ is linear, whence \eqref{eq: ex:vssect} indeed is a commutative diagram in the category of locally convex spaces.

\paragraph{The upper row of diagram \eqref{eq: ex:vssect} is exact.}
In Remark \ref{rem: closedideal} we have seen that $\mathcal{I}_c (C,E)$ is a closed subspace and we know that $\res_C$ is surjective, open and continuous. Hence we only need to prove that its kernel coincides with $\mathcal{I}_c (C,E)$.
Consider $g \in \ker(\res_C)$. 
Since $\rho_\cU$ is injective, the commutativity of \eqref{diag: res} implies that $\rho_\cU (g)$ is contained in the kernel of $q$ and by exactness of the bottom row and the definition of $\rho_\cU|_C$ we must have $g \in \mathcal{I}_c (C,E)$. The converse inclusion is trivial and in conclusion \eqref{eq: ex:vssect} is exact in the category of locally convex spaces. Finally we remark that this implies that $\Gamma_c (C,E)$ is a Hausdorff space (as the quotient of a Hausdorff space modulo a closed linear subspace). 
\end{proof}

An important ingredient in the proof of the last lemma was the local to global argument using a partition of unity. We will see in Lemma \ref{lem: sm:mix} that this construction is continuous with respect to the function space topologies.

\begin{lemma}\label{lemma: r_is_injective}
The map $r \colon \Gamma_c (C,E) \rightarrow \bigoplus_{i} C^\infty (C_i,\R^m)$ is injective and its image is the closed subspace 
$$\mathcal{A} \coloneq \left\{(h_i) \in \bigoplus_{i \in I} C^\infty (C_i,\R^m) \middle| h_i|_{C_i \cap C_j} = \Phi_{ij} (\id_{C_j} , h_j)|_{C_i\cap C_j} \forall i,j \in \N\right\}.$$
If $C$ is compact, $r$ induces an isomorphism $\Gamma_c (C,E) \cong \mathcal{A}$.
\end{lemma}

\begin{proof}
We already know that $r$ is continuous and it is clearly injective and takes its image in $\mathcal{A}$. Now every $C_i \subseteq C$ is
contained in the compact set $\overline{U}_i$. Hence for a family $(f_i) \in \mathcal{A}$, the obvious mapping
$$f \colon C \rightarrow E, \quad f(x) \coloneq T\varphi_i^{-1} f_i (x) , \text{ for } x \in C_i$$
makes sense and is a compactly supported smooth section over $C$, i.e.\ it is contained in $\Gamma_c(C,E)$. Hence $\mathcal{A}$ is the image of $r$. Again since point evaluation and postcomposition by fixed smooth mappings are continuous in the compact open $C^\infty$-topology, \eqref{eq: WSQ1}
shows that this is also the case for the \Frechet topology on $C^\infty (C,\R^m)$. An easy adaptation of the argument in \cite[proof of Lemma 3.21 (b)]{MR3328452} establishes closedness of $\mathcal{A}$.

Let us now assume that $C$ is compact. Then there are only finitely many $i\in I$ such that $C_i \neq \emptyset$ we conclude that $\mathcal{A}$ is a \Frechet space as a closed subspace of a finite product of such spaces. Furthermore, $\Gamma_c (M,E)$ is isomorphic to a closed subspace (cf.\ Lema \ref{lem: top:sect}) of the webbed space $\bigoplus_{i} C^\infty_{\text{co}} (U_i,\R^m)$, whence $\Gamma_c (C,E)$ is webbed as a qotient of a webbed space (\cite[Lemma 24.28]{MR1483073}). Now the open mapping theorem \cite[24.30]{MR1483073} shows that $r$ is open as a mapping into $\mathcal{A}$, whence $r$ induces the claimed isomorphism of locally convex spaces. 
\end{proof}

\begin{remark}
Note that the topology on $\Gamma_c (C,E)$ does not automatically turn $r$ into an isomorphism if $C$ is not compact.
Studying the above proof, the open mapping theorem is not applicable since $\mathcal{A}$ is not necessarily ultrabornological (as it is not clear that it would be a limit subspace of the direct sum).
In fact the authors do not know whether the quotient topology may be properly finer then the one induced by $r$ in the non-compact case.
\end{remark}

However, the problem mentioned in the last remark is not relevant for us, since we will only consider sets which allow continuous extension operators. In the presence of such a section, the two topologies coincide:
 
\begin{lemma}\label{lem: op:ext}
Assume that there exists a continuous section $s \colon \mathcal{A} \rightarrow \Gamma_c (M,E)$ of the map $r\circ \res_C$, then the quotient topology turns $r$ into an isomorphism $\Gamma_c (C,E)\cong \mathcal{A}$.
\end{lemma}

\begin{proof}
Since $r\circ \res_C$ is continuous surjective and admits a (global) continuous section, it is a quotient map between locally convex spaces. As $r^{-1} \circ (r \circ \res_C) = \res_C$ we deduce that $r^{-1} \colon \mathcal{A} \rightarrow \Gamma_c (C,E)$ is continuous, whence $r$ induces and an isomorphism of locally convex spaces onto its image.
\end{proof}

Thus in the situation of Theorem A (to be proved in the end of the section) the topologies coincide.

\subsection*{Interlude: Patching by partition of unity}
\addcontentsline{toc}{subsection}{Interlude: Patching by partition of unity}
In this interlude, we discuss continuity properties for the map which patches mappings on a locally finite-covering by means of a partition of unity. 

\begin{setup}\label{setup: settingup}
Recall that for a given compact subset $K$ of $M$ only finitely many members of 
the locally finite open cover $\cU$ have a non-trivial 
intersection with $K$. Thus for each $i\in I$ we obtain a finite subset of 
$I$ by setting
\begin{displaymath}
 J_i \coloneq \{ j \in I \mid U_j \cap U_i \neq \emptyset\}  
 \end{displaymath}
\end{setup}

\begin{setup}\label{abuse_notation}
Fix $n\in \N$ and consider for $i\in I$ maps $f_j \in C^\infty (U_j,\R^n)$ for $j \in J_i$. Multiplying with the 
partition of unity \ref{setup: prep:ThmA}, we obtain for every such pair a smooth mapping 
$f_{ji} \coloneq \chi_j|_{V_j \cap U_i} \cdot f_j|_{V_j \cap U_i}$ defined on the (possibly empty) set $V_j \cap U_i$. Note that since $\supp \chi_j \subseteq V_j$, the mapping vanishes in a neighborhood of the boundary of $V_j \cap U_i$ in $U_i$. Thus we can extend $f_{ji}$ by $0$ to a smooth map on all of $U_i$ (or by a similar argument to all of $V_i$). In the following we will extend 
these mappings to all of $U_i$ (or similarly to $V_i$) and suppress the extension in the notation. 
\end{setup}

\begin{lemma}\label{lem: sm:mix}
 Using the notation from \ref{abuse_notation}, the mixing map 
 \begin{align*}
  \mu \colon \bigoplus_{i\in I} C^\infty_{co} (U_i,\R^n) & \longrightarrow \bigoplus_{i \in I} C^\infty_{co} (V_i,\R^n) \\
   (f_i)_{i \in I} & \mapsto \left(\sum_{j\in J_i} 
(\chi_j|_{V_i\cap V_j}) \cdot f_j|_{V_{i}\cap V_j}  \right)_{i\in I}
 \end{align*}
 is continuous linear. Its image is contained in the closed subspace $A \coloneq 
\{(g_i)_{i\in I} \mid g_j|_{V_i \cap V_j} = g_i|_{V_i\cap V_j} \forall 
i,j \in I \}$. 
\end{lemma}

\begin{proof}
\textbf{The mapping $\mu$ makes sense.}
As argued in \ref{abuse_notation}, every component of $\mu (f_i)_{i \in I}$ is a smooth function as a finite sum of such functions. 
Note that every $i \in I$ appears only in finitely many of the 
sets $J_k, k \in I$. Thus every $f_i$ appears at most in finitely many of the sums of the definition of $\mu$, whence $\mu$ makes sense as a mapping 
between direct sums. Clearly $\mu$ is linear.

\paragraph{$\mu$ takes its image in $A$}
By construction we have $\supp \chi_k \subseteq V_k$. Hence if $\chi_k$ 
does not vanish on $V_i \cap V_j$ we must have $k \in J_i \cap J_j$. 
Thus the sum in $\sum_{k\in J_i} (\chi_k|_{V_k \cap V_i}) \cdot f_k|_{V_i \cap V_k}$ 
coincides on $V_i \cap V_j$ (up to vanishing 
summands) with the one summing over $J_j$. In conclusion, 
$\mu$ takes its image in $A$.

\paragraph{Continuity of the auxiliary mappings $m_i$}
Let us first 
fix $i\in I$ and consider the linear map
\begin{align*}
m_i \colon \bigoplus_{j \in 
J_i} C^\infty_{co} (U_j,\R^n) & \rightarrow C^\infty (V_i,\R^n)\\
(f_j)_j &\mapsto \sum_{j\in J_i} (\chi_j|_{V_i\cap V_j}) \cdot f_j|_{V_i \cap V_j} .
\end{align*}
As $J_i$ is finite and $C^\infty_{co} (V_i , \R^n)$ is a 
topological vector space it clearly suffices to establish smoothness for all of the mappings
\begin{align*}
  c_j\colon C^\infty_{co} (U_j,\R^n) &\rightarrow C^\infty_{co} (V_i, \R^n),\quad  j \in J_i\\ 
  f &\mapsto \chi_j|_{V_i\cap V_j} \cdot f|_{V_i \cap V_j}.
\end{align*}
Recall that the space $C^\infty_{co} (U_j,\R^n)$ is a topological $C^\infty_{co} (U_j,\R)$-module (see e.g.\ \cite[Corollary F.13]{math/0408008v1}). Thus the map $\kappa_j (f) \coloneq \chi_j|_{U_j} \cdot f$ is continuous, takes its image in the linear subspace $C^\infty_{\overline{V}_j} (U_j,\R^n) \subseteq C^\infty_{\text{co}} (U_j,\R^n)$ of smooth functions supported in $\overline{V}_j$.
Now \cite[Lemma 4.24 and Lemma 4.6]{math/0408008v1} 
extending functions in $C^\infty_{\overline{V}_j} (U_j, \R^n)$ by $0$ to all 
of $M$ and restricting then to $V_i\cap V_j$ yields a continuous linear map which, composed with $\kappa_j$, coincides with $c_j$. We conclude that $c_j$ and thus $m_i$ is continuous linear.

\paragraph{Continuity of $m$}
We define the mapping 
\[
\tilde{\mu} \coloneq \oplus_{i\in I} m_i \colon \bigoplus_{i \in I} \bigoplus_{j \in J_i} C^\infty_{co} 
(U_i ,\R^n) \rightarrow \bigoplus_{i \in I} C^\infty_{co} (V_i ,\R^n).
\]
This mapping is continuous linear, since the 
mappings $m_i$ are so by the previous step.
It follows from the universal property of the locally convex direct sum that 
\begin{align*}
 B_{diag} \colon \bigoplus_{i \in I} C^\infty (U_i, \R^n) &\rightarrow 
\bigoplus_{i\in I} \bigoplus_{\# \{k \in I \mid i \in J_k\}} C^\infty (U_i,\R^n)\\ & \cong 
\bigoplus_{i\in I} \bigoplus_{j \in J_i} C^\infty (U_j,\R^n)\\
(f_i)_{i\in I} &\mapsto\quad ( \oplus_{\# \{k \in I \mid i \in J_k\}} f_i)_{i\in I} \\ & \mapsto \quad ((f_j)_{j\in J_i})_{i\in I}
\end{align*}
is continuous linear (where due to the construction, there is a 
bijection between the index sets of both sums).
Now, we have $\mu = \tilde{\mu} \circ B_{diag}$ and thus $m$ is continuous linear as a composition of such mappings.
\end{proof} 

\subsection*{Global extensions of bundle sections on a closed set}
\addcontentsline{toc}{subsection}{Global extensions of bundle sections on a closed set}

We will now prove Theorem A from the introduction, whose statement we repeat here for convenience.

\begin{theorem}[Theorem A]\label{thmA}
Let $C\subseteq M$ be a closed set satisfying the cusp condition. Then the restriction map $\res_C \colon \Gamma_c (M,E) \rightarrow \Gamma_c (C,E)$ admits a continuous linear section $\mathcal{E}_C^M$.
\end{theorem}

\begin{proof}
We use the notation and data introduced in \ref{setup: prep:ThmA}.
For the proof we consider a commutative diagram of locally convex spaces (where the numbers indicate where the (continuous) linear map was constructed):

\begin{equation} \begin{aligned}
 \begin{xy}
  \xymatrix{
      \Gamma_c (C,E) \ar[r]^-{r}_-{\eqref{eq: ex:vssect}} \ar@/^1pc/@{-->}[dd]^{\mathcal{E}_C^M} &  \bigoplus_{i \in I} C^\infty (C_i)^m\ar[rr]^-{\text{extension}}_-{\text{Lemma \ref{lem: multiFrerick}}}  & & \bigoplus_{i \in I} C^{\infty}_{\text{co}} (U_i)^m  \ar[ldd]|-{\stackrel{\mu|^{\im(\rho_\cV)}}{\text{Lemma }\ref{lem: sm:mix}}} \ar[dd]^{\mu} \\
                &   &   &   \\
 \Gamma_c (M,E) \ar[uu]^{\res_C}   \ar[rr]^{\rho_\cV}   && \im(\rho_\cV) \ar[r]^-{\text{inc}} & \bigoplus_{i \in I} C^\infty_{\text{co}} (V_i,\R^m) 
  }
\end{xy}\end{aligned}\label{diag: l:con}
\end{equation}

We postpone the proof of the commutativity of \eqref{diag: l:con} to Lemma \ref{lem: multiFrerick} below, where also the extension map needed in the computation is defined.

Now $\mathcal{E}_C^M$ is defined via the right half of the diagram (using that $\rho_\cV$ is a topological embedding onto its image by Lemma \ref{lem: top:sect}). Since all the mappings in the definition are continuous and linear, $\mathcal{E}_C^M$ is a continuous linear section of $\res_C$.
\end{proof}

\begin{lemma}\label{lem: multiFrerick}
There exists a continuous linear extension map 
\[
\varep \colon \bigoplus_{i \in I}C^\infty(C_i)^m \rightarrow \bigoplus_{i\in I} C^\infty_{\text{co}} (\R^d)^m
\] 
which makes \eqref{diag: l:con} commutative.
\end{lemma}

\begin{proof}
To construct $\varep$, we first construct continuous linear mappings $\varep_i \colon C^\infty (C_i)^m \rightarrow C^\infty_{\text{co}} (\R^d)^m$ and set $\varep \coloneq \oplus_{i \in I} \varep_i$. Thus $\varep$ will be continuous linear by the properties of the direct sum. For the construction we distinguish two cases depending on $i\in I$:

\paragraph{Case 1: $C_i = \emptyset$.} Since the chart does not intersect the domain of our map, only have to extend the empty function, whence $E_i$ is simply defined as the constant $0$-map in this case.

\paragraph{Case 2: $C_i\neq \emptyset$.} Due to our setup, the sets $C_i$ satisfy the assumptions made in the statement of Lemma \ref{lemma:splitting restriction in charts}. Hence in this case there is a continuous linear extension operator $C^\infty (C_i) \rightarrow C^\infty (U_i)$. We define $\varep_i$ as the $m$-fold product of this extension operator. 

This completes the construction of the extension map and all there is left is to prove that the diagram \eqref{diag: l:con} commutes. 
However, this is obvious from a trivial calculation if one recalls the following facts:
\begin{itemize}
\item For each pair $i,j \in I$ and $(f_k)_k \in \im r$ we have $f_i|_{C_i \cap C_j} = f_j|_{C_i \cap C_j}$,
\item the extension operators $\varep_i$ do not change the map on the $C_i$,
\item composition with $\rho_V^{-1}\mu$ is just mixing and restricting with a partition of unity and then reconstruction via the sheaf property of  smooth maps.
\end{itemize}
Composing again with the restriction map, the outer square of \eqref{diag: l:con} commutes.
\end{proof}

As a direct consequence of the above theorem, we obtain:

\begin{corollary}\label{splitting consequences}
If $C\subset M$ is a closed subset which satisfies the cusp condition in the geodesic metric on $M$, the exact sequence \eqref{eq: ex:vssect} splits, i.e.\ we have the isomorphim of topological spaces $\Gamma_c (M,E) \cong \mathcal{I}_c (C,N) \oplus \Gamma_c (C,N)$.
\end{corollary}

\begin{remark}
As a consequence of Theorem \ref{thmA} and Lemma \ref{lem: op:ext} the map $r$ from Proposition \ref{prop: Whit:mfd} and \eqref{diag: l:con} is a topological embedding onto a closed subspace for every 
closed subset which satisfies the cusp condition.
\end{remark}

\section{Manifolds of mappings for manifolds with rough boundary}
 \label{sect: mfdmap:rb}
 
 In this section we recall some essentials on manifolds with rough boundary from \cite[Chapter 1.4]{GloecknerNeebBuch}. 
 Then we recall the classical construction of manifolds of mappings and how to apply them to the rough boundary case.
 
 \begin{definition}\label{defn: RBM}
We recall from \cite{GloecknerNeebBuch} (cf.\ \cite[Section 4]{alas2012}) that a \emph{manifold with rough boundary} modelled on a locally convex space~$E$ is a Hausdorff topological space $M$ with an atlas of smoothly compatible
homeomorphisms $\phi\colon U_\phi\to V_\phi$ from open subsets $U_\phi$ of $M$ onto locally convex subsets $V_\phi\subseteq E$ with dense interior (to distinguish from ordinary manifold charts, they are also called \emph{rough} $E$-charts).

If $x\in M$ we call $x$ a \emph{formal boundary point} if there is a rough $E$-chart $(U_\varphi, \varphi)$ around $x$ such that $\varphi(x) \in \partial \varphi(U_\varphi)$. Denote by $\partial M$ the \emph{(formal) boundary} of $M$, i.e.\ the set of formal boundary points of $M$. 

If each $V_\phi$ is open, $M$ is an ordinary manifold (without boundary). If each $V_\phi$ is relatively open in a closed hyperplane $\lambda^{-1}\!\left([0,\infty[\,\right)$, where $\lambda\in E'$ (the space of continuous linear functional on $E$), then $M$ is a \emph{manifold with smooth boundary}. In the case of a \emph{manifold with corners}, each $V_\phi$ is a relatively open
subset of $\lambda^{-1}_1\!\left([0,\infty[\,\right)\cap\cdots\cap \lambda^{-1}_n\!\left([0,\infty[\,\right)$, for suitable $n\in \N$ (which may depend on $\phi$) and linearly independent $\lambda_1,\ldots,\lambda_n\in E'$.
\end{definition}

The boundary of manifolds with rough boundary is characterised by the following.   

\begin{lemma}\label{lem: rb:boundary}
Let $M$ be a manifold with rough boundary and $(U_\varphi, \varphi)$, $(U_\psi,\psi)$ be rough $E$-charts around $x\in M$. Then $\varphi (x) \in \partial \varphi(U_\varphi)$ if and only if $\psi (x) \in \partial \psi (U_\psi)$.\footnote{A full proof is contained in the forthcoming \cite[Section 3]{GloecknerNeebBuch}. However here is a rough sketch: Argue by contradiction. In the chart where the image is in the boundary choose a convex neighborhood $W$. Now apply the Hahn-Banach theorem to separate the image of $x$ from the interior of $W$ by a functional $\lambda$. Taking the derivative $\kappa$ of the change of charts, one derives a contradiction by considering $\lambda \circ  \kappa^{-1}$ on $\kappa (W)$.}
\end{lemma}

In essence Lemma \ref{lem: rb:boundary} shows that the formal boundary of $M$ arises from the topological boundary of the images of charts in the model space. 

\begin{remark}
By virtue of the chain rule Lemma \ref{lem: reschain} 2.\ we can define smooth mappings on manifolds with rough boundary in the usual way.

Direct products of manifolds with rough boundary, tangent spaces and tangent bundles\footnote{The definition of tangent vectors at $x$ as equivalence classes of vectors in the model space, i.e.\ $(\psi,x,v) \sim (\varphi,x,w)$ if and only if $d(\psi\circ \varphi^{-1}) (\varphi (x),v)=w$ where $\varphi,\psi$ are manifold charts, makes sense for a manifold with rough boundary and yields  the usual results and identifications for iterated tangent bundles, cf.\ also \cite[Section 2]{michor}.} as well as vector bundles may be defined as usual.

\end{remark}  
 
 We recall the compact open $C^\infty$-topology on the space $C^\infty (M,N)$ of smooth mappings from a manifold with rough boundary to a manifold without boundary.
 
 \begin{definition}\label{defn: topo:init}
 Let $M,N$ be manifolds with rough boundary.
 We define the compact-open $C^\infty$-topology on $C^\infty (M,N)$  as the initial topology induced by the mappings
  \begin{align*}
  T^k \colon C^\infty (M,N) & \rightarrow C_{\text{co}}(T^k M,T^k N)\quad \quad k \in \N_0,\\
  f & \mapsto T^k f
 \end{align*}
 where the right hand side carries the compact open topology. Denote by $C^\infty_{\text{co}} (M,N)$ the space with the compact open $C^\infty$-topology. 
\end{definition}

 \begin{remark} \label{rem: coinfty}
  If $N=\R^m$ the compact open $C^\infty$ topology coincides with topology described in \cite[Definition 3.21 and Section 4]{alas2012} by adapting the argument of \cite[Lemma 1.14]{MR1911979} to manifolds with rough boundary.
  In addition, if $M$ is a locally convex regular closed subset of $\R^m$ and $N=\R^m$, the compact open $C^\infty$-topology coincides with the identification topology from \ref{setup: WJ} (by an argument analogous to \ref{Whitney:fun}).
\end{remark}
 
 \begin{proposition}\label{prop: fcomp:cont}
  Let $f \colon N\rightarrow B$ and $g \colon A \rightarrow M$ be smooth mappings between finite-dimensional manifolds with rough boundary. Then
  \begin{align*}
  f^* g_* \colon C^\infty_{\text{co}}(M,N) & \rightarrow C^\infty_{\text{co}} (A,B) \\
   h & \mapsto f\circ h \circ g 
  \end{align*}
  is continuous.
 \end{proposition}
 
 \begin{proof}
 The usual proof for manifolds without boundary (see e.g.\ \cite[Lemma 5.5.]{Wockel13}) carries over without any changes.
 \end{proof}

 \begin{lemma}\label{top: embedding}
 The initial topology turns the map 
 \begin{align*}
 \mathcal{T} \colon C^\infty_{\text{co}} (M,N)& \rightarrow \prod_{k \in\N_0} C_{\text{co}}(T^k M,T^k N)\\ 
 f& \mapsto (T^k f)_{k \in \N_0}
 \end{align*}
 into a topological embedding with closed image.
 \end{lemma}
 
 \begin{proof}
 By definition of an initial topology, the map $\mathcal{T}$ is a topological embedding. Let now $(T^kf^\alpha)_{k\in \N_0, \alpha \in I}$ be a net in the image of $\mathcal{T}$ which converges to $(f_k)_{k \in \N_0}$.
 If we can prove that $f_k = T^k f_0$ holds for every $k\in \N$ then the image of $\mathcal{T}$ is closed. Clearly we can verify the formula $f_k = T^k f_0$ locally in charts. As the sequence converges with respect to the compact open $C^\infty$-topology, the usual inductive proof using the fundamental theorem of calculus (\cite[Theorem 1.5 and Lemma 1.7]{MR1911979} which is valid on locally convex regular subsets of $\R^d$!) carries over without any changes, see e.g.\ \cite[Lemma 5.13 and Theorem 5.14]{Wockel13}.
 \end{proof}

As a consequence of the above identification, we obtain the following completeness and metrisation results (which are well known in the case of a manifold without boundary).
 
 \begin{lemma}\label{lem: met:top}
 Let $M,N$ be manifolds with rough boundary, such that $M$ is locally compact\footnote{In contrast to manifolds without boundary, manifolds with rough boundary need not be locally compact. For example, recall that regular locally convex subsets of $\R^d$ are in general not locally compact, e.g.\ $\{(0,0)\} \cup \{(x,y) \in \R^2 \mid x >0\}$ in $\R^2$.} and $\sigma$-compact and $N$ metrisable and modelled on a metrisable space. Then $C^\infty_{\text{co}} (M,N)$ is metrisable.
 \end{lemma}
 
 \begin{proof}
  For convenience let $d$ be the dimension of $M$. Since $M$ is locally compact, for every manifold chart $(\varphi, U)$ the domain $U$ is locally compact, whence $\varphi (U)$ is locally compact subset of $\R^d$. As $\R^d$ is second countable, $\varphi (U)$ is locally compact and second countable, whence $\sigma$-compact. 
Using the canonical atlas for the iterated tangent bundle, we see that locally over $U$ we get a bundle trivialisation $T^k M \supseteq T^k U \cong \varphi (U) \times (\R^d)^{2^k-1}$. Thus $T^kM$ is $\sigma$-compact for every $k \in \N_0$.
Since $N$ is metrisable and modelled on a metrisable space, $T^kN$ is metrisable by \cite{EG54}. Thus the spaces $C_{\text{co}}(T^kM,T^kN)$ are metrisable by \cite[Proposition 5.10 e)]{Wockel13} whence the embedding from Lemma \ref{top: embedding} indentifies $C^\infty_{\text{co}} (M,N)$ as a subspace of a metrisable space. 
 \end{proof}
 
 \begin{corollary}\label{cor: complete}
 Let $M$ be a locally compact manifold with rough boundary and $F$ a \Frechet space. Then $C^\infty_{\text{co}} (M,F)$ is a \Frechet space.
 \end{corollary}
 
 \begin{proof}
 It is well known that the pointwise operations turn $C^\infty_{\text{co}} (M,F)$ into a topological vector space (cf.\ \cite[Section 2]{alas2012}).
 By Lemma \ref{lem: met:top} the space $C^\infty_{\text{co}} (M,F)$ is metrisable. In addition, \cite{EG54} implies that every $T^kM$ is metrisable, whence a $k$-space. Now $T^kF\cong F^{2^k}$ is complete and we infer from \cite[Remark 3.2 (a)]{MR1934608} that $C_{\text{co}} (T^kM,T^kF)$ is complete. Now Lemma \ref{top: embedding} identifies $C^\infty_{\text{co}} (M,F)$ with a closed subspace of the complete space $\prod_{k \in \N_0} C_{\text{co}} (T^kM,T^kF)$, whence $C^\infty_{\text{co}} (M,F)$ is a \Frechet space. 
 \end{proof}
 
 Finally we turn to smooth sections of bundles over a manifold with rough boundary.
 
 \begin{setup}\label{setup: sect:loc}
 Let $p \colon E \rightarrow M$ be a vector bundle with typical fibre $F$.
 Assume that $F$ is a \Frechet space and $M$ is a compact manifold with rough boundary. Then we define 
 $$\Gamma (M,E) \coloneq \{\sigma \in C^\infty (M,E) \mid p \circ \sigma = \\id_M\},$$
 and endow it with the subspace topology induced by $C^\infty_{\text{co}} (M,E)$. 
 
 Note that since $M$ is compact, we can choose an open cover $(W_i,\kappa_i)_{1\leq i \leq n}$ of domains of bundle trivialisations for $E$ and denote by $E|_{U_i}$ the restricted bundle over $U_i$. Then define the map 
 \begin{align*}
 \rho \colon \Gamma (M,E) & \rightarrow \prod_{1 \leq i \leq n} \Gamma (W_i,E|_{W_i)})\\ 
 \sigma & \mapsto (\sigma|_{W_i})_i.
 \end{align*}
 By Lemma \ref{prop: fcomp:cont} the map $\rho$ is continuous (as each of its components are given by mappings $(\iota_i)^*$, where $\iota_i \colon W_i \rightarrow M$ is the inclusion). Clearly $\rho$ is injective, linear and identifies $\Gamma (M,E)$ with the subspace $\mathcal{C}\coloneq \{(\gamma_i)_i \mid \gamma_i|_{W_i \cap W_j} = \gamma_j|_{W_i\cap W_j}\}$. Since the the evaluation map is continuous by \cite[Proposition 3.20]{alas2012}, $\mathcal{C}$ is closed. Working with subbasic neighborhoods, one can also prove that $\rho$ is open onto its image. We refer to \cite[Lemma 6.4]{Wockel13} for details.
 
 Now we obtain isomorphisms $\Gamma (W_i , E|_{W_i}) \rightarrow C^\infty_{\text{co}} (W_i,F),\ \sigma \mapsto \pr_2 \circ \psi_i \circ \sigma$ of topological vector spaces.
 Thus Corollary \ref{cor: complete} implies that $\Gamma (M,E)$ is a \Frechet space as a closed subspace of a direct product of \Frechet spaces.
 \end{setup}

Having discussed the topology on the manifold of mappings, we will now construct an infinite-dimensional manifold structure on $C^\infty_{\text{co}} (M,N)$ for $M$ a compact manifold with rough boundary and $N$ a manifold without boundary. 
If $M$ is a manifold with corners, such a construction (even for $M$ non-compact) can be found in \cite{michor}. Our proof follows the general idea of loc.~cit.\ but we avoid using an instance of the so called $\Omega$-Lemma. 
For the rest of this section $M$ will be a compact manifold with rough boundary and $N$ will be a \Frechet manifold without boundary which admits a local addition.

\begin{definition}
A manifold $N$ admits a local addition, if there is a \emph{local addition} $\Sigma_N$ on $N$, i.e.\ a smooth mapping $\Sigma_N \colon TN \supseteq \Omega \rightarrow N$  on an open $0$-neighborhood $\Omega$, such that 
\begin{enumerate}
\item $\Sigma_N \circ 0_N = \id_N$, where $0_N \colon N \rightarrow TN$ is the $0$-section.
\item $(\pi_{N}|_{\Omega}, \Sigma_N) \colon \Omega\rightarrow N \times N$ induces a diffeomorphism onto an open neighborhood of the diagonal in $N \times N$.
\end{enumerate}
\end{definition} 

Using the local addition, we obtain canonical charts for the mapping space:

\begin{setup}[Canonical charts]\label{setup: can:charts}
For $f \in C^\infty (M,N)$ we let $f^*TN$ be the pullback bundle of the tangent bundle of $N$ with associated bundle map $F\colon f^*TN \rightarrow TN$.
Via the bundle map $F$, we can identify the \Frechet space of sections $\Gamma (M,f^*TN)$ (cf.\ \ref{setup: sect:loc}) with 
$$C_f^\infty (M,TN) \coloneq \{g \in C^\infty (M,N) \mid \pi_N \circ g = f\} \subseteq C^\infty_{\text{co}} (M,TN)$$
in the subspace topology. In the following we will suppress this (harmless) identification without further notice.
Then define the sets 
\begin{align*}
V_f & \coloneq \{g \in C^{\infty} (M,E) \mid g(M) \subseteq \Omega\} \cap C^\infty_f (M,TN), \\
U_f & \coloneq \{g \in C^\infty (M,N) \mid (f,g)(M) \subseteq (\pi,\Sigma_N) (\Omega)\} \subseteq C^\infty_{\text{co}} (M,N)
\end{align*}
and note that both are open in the compact open $C^\infty$-topology. 
Now Proposition \ref{prop: fcomp:cont} implies that $\varphi_f \colon U_f \rightarrow W_f, g \mapsto (\pi_N, \Sigma_N)^{-1} \circ (f,g)$ is a homeomorphism with inverse $\varphi_f^{-1} (\tau) = (\Sigma_N)_* (\tau)$.
\end{setup}

\begin{theorem}\label{thm: mm:rb}
Let $M$ be a compact manifold with rough boundary and $N$ a metrisable manifold modelled on a \Frechet space which admits a local addition.
The atlas of canonical charts $(U_f,\varphi_f)_{f\in C^\infty (M,N)}$ turns $C^\infty_{\text{co}} (M,N)$ into a \Frechet manifold modelled on spaces of sections $\Gamma (M, f^*TN)$.
The manifold structure does not depend on the choice of local addition.
\end{theorem}

\begin{proof}
The compact open $C^\infty$-topology on $C^\infty (M,N)$ is clearly finer than the compact open topology, whence $C^\infty_{\text{co}} (M,N)$ is Hausdorff.
In \ref{setup: can:charts} we have already seen that the canonical charts form an atlas of homeomorphisms. Hence we only have to prove that the change of charts are smooth. 

\paragraph{Change of charts formula.} To this end, observe that $h \coloneq \varphi_f \circ \varphi_g^{-1}$ has an open domain, $O_{f,g} \subseteq \Gamma (M,g^*TN)$ for each pair $f,g \in C^\infty (M,N)$.
Let now $\tau$ be in $O_{f,g}$ and $x\in M$, then we obtain the formula
\begin{equation}\label{eq: chch:form}
h \colon O_{f,g} \rightarrow \Gamma (M,f^*TN),\quad h(\tau)(x) = (\pi_N, \Sigma_N)^{-1} (f(x) , \Sigma_N \circ \tau (x))
\end{equation}
\paragraph{Localisation in charts}
Choose an atlas $(W_i,\kappa_i^f)_{1\leq i\leq n}$ of local trivialisations for the bundle $f^*TN$. Adjusting our choices if necessary, we may assume that for each $1\leq i\leq n$ $W_i$ is \begin{enumerate}
\item a domain of a bundle trivialisations $(W_i,\kappa_i^g)$ of $g^*TN$ 
\item the domain of a manifold chart $(W_i,\psi_i)$ of $M$.
\end{enumerate}
Now as in \ref{setup: sect:loc} we construct the topological embedding $\rho \colon \Gamma (M,f^*TN) \rightarrow \prod_{1\leq i\leq n} C^\infty_{\text{co}} (W_i,F)$, where $F$ is the model space of $N$.
Since precomposition with a smooth is continuous in the compact open $C^\infty$-topology by Proposition \ref{prop: fcomp:cont}, we see that $C^\infty_{\text{co}} (W_i,F) \cong C^\infty_{\text{co}} (\psi_i (W_i),F)$ as locally convex spaces (similarly for $g^*TN$). 
We remark for later use that since $M$ is compact, $W_i$ and also $\psi_i (W_i)$ are locally compact. 

\paragraph{Smoothness via the exponential law.} 
Now $\rho$ is a topological embedding with closed image, whence $h$ will be smooth if and only if $\rho \circ h$ is smooth and this is the case if and only if for each $1\leq i \leq n$ the mapping 
\begin{align*}
h_i \colon O_{f,g} & \rightarrow C^\infty (W_i,F)\\ 
\tau & \mapsto \pr_2 \circ \kappa_i \circ h(\tau)|_{W_i}
\end{align*}
is smooth.
Now we recall that $O_{f,g} \subseteq \Gamma (M,g^*TN)$ is open and the spaces $\Gamma (M,g^*TN)$ and $C^\infty (W_i,F)$ are \Frechet spaces by \ref{setup: sect:loc}. Since $W_i$ is a manifold with rough boundary (being an open subset of $M$) the exponential law \cite[Theorem B]{alas2012} for smooth mappings on manifolds with rough boundary yields: The $h_i$ (and thus $h$) are smooth if and only if the mapping 
$$h_i^\vee \colon O_{f,g} \times W_i \rightarrow F,\quad h_i^\vee (\tau,x) \coloneq h_i (\tau)(x)$$
is smooth. However, \eqref{eq: chch:form} allows us to write 
\begin{equation}\label{eq: chch:exp}
h_i^\vee (\tau) (x) = \pr_2 \circ (\pi_N, \Sigma_N)^{-1} \circ (f, \Sigma_N \circ (\ev(\tau|_{W_i})) (x),
\end{equation}
where $\ev \colon \Gamma (W_i, g^*TN|_{W_i}) \times W_i \rightarrow g^*TN|_{W_i},\ (f,x)\mapsto f(x)$ is the evaluation map. 
Using that $g^*TN|_{W_i}$ is trivial we identify $\Gamma (W_i, g^*TN|_{W_i}) \cong C^\infty_{\text{co}} (\psi_i (W_i), F)$ and deduce from \cite[Proposition 3.20]{alas2012} that $\ev$ is smooth. As the restriction of $\tau$ is smooth by \ref{setup: sect:loc}, we deduce from \eqref{eq: chch:exp} that $h_i^\vee$ is smooth as a composition of smooth functions. 
Summing up the change of charts are smooth and the canonical charts form indeed a smooth atlas turning $C^\infty (M,N)$ into a \Frechet manifold

\paragraph{The construction is independent of the choice of local addition.} Replacing the local addition $\Sigma_N$ by $\tilde{\Sigma}_N$, the change of charts formula \eqref{eq: chch:form} shows that the change of charts between a chart with respect to $\Sigma_N$ and with respect to $\tilde{\Sigma}_N$ will be smooth. 
Hence the manifold structure does not depend on the choice of local addition.
\end{proof}

\begin{remark}
 A crucial ingredient in the proof of Theorem \ref{thm: mm:rb} was compactness of $M$ to endow the function space $C^\infty (M,N)$ with the compact open $C^\infty$ topology and to use the exponential law instead of the so called $\Omega$-Lemma \cite[8.7]{michor}. 
 Though the authors believe that for non-compact $M$, the space $C^\infty (M,N)$ can be endowed with a manifold structure along the lines described in \cite{michor}, this would involve two significant steps: One has to define a version of the fine very strong topology for mapping spaces on non-compact manifolds of mappings and reprove the results outlined in \cite{HS17} (mostly trivial with the notable exception of continuity of the composition). 
 Then one needs an analogue of the $\Omega$-Lemma for manifolds with rough boundary (which will be contained in \cite{GloecknerNeebBuch}, due to H.\ Gl\"{o}ckner, private communications). 
\end{remark}

\section{Submanifolds with rough boundary and the proof of Theorem B}\label{sect: MFDMAP}

In this section we establish the global version of the splitting of spaces of compactly supported sections. 
Our aim is to construct submersions between the infinite-dimensional manifolds of mappings. 
To this end we need to clarify first our concept of a submanifold with boundary sitting inside of manifolds without boundary. If we require the submanifold with rough boundary to be closed, then it will automatically have no narrow fjords. The authors believe that they will also automatically satisfy the cusp condition but were not able to prove the latter statement.

\begin{definition}[Submanifold with rough boundary]
Let $M$ be a finite-dimensional manifold (possibly with rough boundary). A subset $S \subseteq M$ is called \emph{(embedded) submanifold with rough boundary} of $M$ if for every $p \in S$ there is a chart $(U_p,\varphi_p)$ of $M$ with $p\in U_p$ and $\varphi_p (p)=0$ and a regular locally convex subset $R_p \subseteq \varphi_p(U_p)$ such that $\varphi_p (S \cap U_p) = R_p$. 
If for every $p \in S$ the regular locally convex set is a relatively open set in a quadrant $[0,\infty[^m \times \R^{d-m}$ (cf.\ Definition \ref{defn: RBM}), then we say that $S$ is an (embedded) \emph{submanifold with corners}.
If in addition $S$ is a closed subset which satisfies the above conditions, we say that $S$ is a \emph{closed submanifold with rough boundary} (or \emph{with corners}, respectively).  
\end{definition}

\begin{remark}
\begin{enumerate}
\item A submanifold with rough boundary inherits the structure of a manifold with rough boundary from the ambient manifold and this structure turns the inclusion $\iota_S \colon S \rightarrow M$ into a smooth embedding.
Thus submanifolds with rough boundary as defined here are initial submanifolds, i.e.\ a mapping $f \colon N \rightarrow S \subseteq M$ between manifolds with rough boundary is smooth as a map to $S$ if and only if it is smooth as a map to $M$. 
Note that closed submanifolds of $\sigma$-compact manifolds are again $\sigma$-compact.
\item We remark that our definition of an embedded submanifold with corners is a special case of a submanifold with corners as in \cite[2.5]{michor}. 
Since we are only interested in a very specialised case, we do not need the more general definition. 
In particular, we refrain from defining submanifolds of lower dimension (which could be done as usual but is not needed here).
\item Due to our definition an embedded submanifold with rough boundary $S\subseteq M$ is regular: Consider $x\in \partial S$ and let $(U_\varphi,\varphi)$ be a submanifold chart. 
Then $\varphi(x) \in \varphi (U_\varphi \cap S) = \varphi(U_\varphi) \cap C$ for a regular locally convex set $C$. 
Hence $W \coloneq \varphi^{-1}(\varphi (U_\varphi \cap C^\circ)) \subseteq S$ is an open set in $M$, whence contained in the interior of $S$. 
Choosing a sequence in W we can approximate $\varphi (x)$, whence $x \in \overline{S^\circ}$. 
Thus $S$ is regular and we see in addition that $x \in \partial S$ entails $\varphi (x) \in \partial \varphi(U_\varphi)$.  
\end{enumerate}  \label{rem: submfd:rb}
\end{remark}

Before we continue, let us construct a class of examples for submanifolds with rough boundary of a Riemannian manifold which will be used to prove Corollary D from the introduction.

\begin{lemma}\label{lem: sconv:rbsmfd}
Let $M$ be a Riemannian manifold and $C$ be a regular closed subset which is strongly convex, i.e.\ for every $p,q \in C$ exists a unique minimal geodesic segment $\overline{pq}$ connecting $p$ and $q$ such that $\overline{pq} \subseteq C$.
Then $C$ is a submanifold of $M$ with rough boundary.
\end{lemma}

\begin{proof}
By standard Riemannian geometry, we can choose for every $p\in C$ an open $0$-neighborhood $W_p \subseteq T_pM$ such that the restriction of the Riemannian exponential map $\exp_p \coloneq \exp|_{W_p} \colon  W_p \rightarrow M$ induces a diffeomorphism onto an (open) $p$-neighborhood in $M$. 
We will show that the manifold charts $\exp_p$ induce suitable submanifold charts with rough boundary.

Clearly if $p \in C^\circ$ we can just shrink $W_p$ to obtain such a submanifold chart. For $q \in \partial C$ (the boundary of $C$) we have to work harder. 
Define for $q \in \partial C$ the set
$$V_q \coloneq \{\lambda w \in T_q M \mid 0< \lambda < \infty, w\in W_q \text{ and } \exp_q (w) \in C^\circ\}.$$
We observe that $V_q = \bigcup_{0 < \lambda < \infty }\lambda \cdot \exp_q^{-1} (C^\circ)$ is an open subset of $T_q M$.
Now we exploit the geometric properties of strongly convex sets following \cite{MR0226542}, where these sets are called "konvex".\footnote{Loc.cit.\ assumes that $M$ is a complete Riemannian manifold. 
We do not assume completeness as the parts of \cite{MR0226542} needed here do not rely on the completeness of $M$. In fact, the geodesic segments needed in the proofs exist since we are working in the strongly convex set $C$.}
As $C$ is strongly convex, also the interior $C^\circ$ is strongly convex \cite[Korollar 4.5.1]{MR0226542}, whence it is also "schwach konvex" (weakly convex) in the terminology of loc.~cit.. 
Now regularity of $C$ implies that $q \in \partial C = \partial C^\circ$. 
We can thus copy the argument in the proof of \cite[4.9.2]{MR0226542} verbatim (note that the geodesics occuring there are only needed locally in a small neighborhood around $q$!) to establish that $V_q$ is a convex cone in $T_qM$ whose tip is $0_q$. 
In particular $\overline{V}_q$ is a convex cone, i.e.\ a closed subset with dense interior that is (locally) convex.

\paragraph{Claim:} $\exp_q (W_q \cap \overline{V}_q) = C \cap \exp_q (W_q)$.
If this is true then $\exp_q$ restricts to a submanifold chart (with rough boundary) for $C$ around $q$ as $W_q \cap \overline{V}_q$ has dense interior (namely $W_q \cap V_q$) and is locally convex as an intersection of two (locally) convex sets in $T_pM$.
We conclude that $C$ is a closed embedded submanifold with rough boundary of the Riemannian manifold $M$.

\paragraph{Proof of the claim:}
 Observe that $\exp_q (V_q \cap W_q) = \exp_q (W_q) \cap C^\circ$ by construction of $V_q$ (and the diffeomorphism property of $\exp_q$). 
Now let $x \in W_q \cap \overline{V}_q$. Since the interior of this subset is the dense set $W_q \cap V_q$, we can choose and fix a sequence $(x_n)_{n\in \N} \subseteq  V_q \cap W_q$ with $\lim_{n} x_n = x$. By continuity of $\exp_q$ and since $\exp_q (x_n) \in C^\circ \cap \exp_q (W_q)$ we have 
$\exp_q (x) \in \exp_q (W_q) \cap \overline{C^\circ} = \exp_q (W_q) \cap C$. Conversely, if $p_n \in C^\circ \cap \exp_q (W_q)$ is a sequence converging to $p \in C \cap \exp_q (W_q)$ we use continuity of $\exp_q^{-1}$ to see that $p \in \exp_q (W_q \cap \overline{V}_q)$.
Summing up, the claim follows.
\end{proof}

\begin{remark}\label{rem: boundary: strong:conv}
Actually, strongly convex subsets have very nice boundary behaviour. For example, it is known that they have Lipschitz boundary (cf.\ e.g.\ \cite{MR638797}). 
However, we are not aware of another source in the literature where submanifold charts of the above kind are explicitely constructed. 

In light of Example \ref{ex: cusp} this implies that a strongly convex regular closed subset satisfies the cusp condition.
\end{remark}

Encouraged by these results, we shall now prove that every closed submanifold with rough boundary satisfies the cusp condition, Definition~\ref{def:cusp_condition}. 

\begin{proposition}\label{prop: rsub:nnf}
Let $M$ be a Riemannian manifold and $C$ be a closed submanifold with rough boundary. Then $C$ satisfies the cusp condition.
\end{proposition}

\begin{proof}
Following Remark \ref{rem: submfd:rb} we already know that $C$ is a regular closed set. We first have to check the no narrow fjord condition. To this end, fix $x\in C$ together with a manifold chart $\varphi \colon U \rightarrow \varphi (U) \subseteq \R^n$ around $x$ and $\varphi (U\cap C) = \varphi(U)\cap R$ for a suitable regular locally convex set $R$.
Shrinking the chart if necessary, we can assume that it is bi-Lipschitz with respect to the geodesic length metric and the Euclidean metric (cf.\ Remark~\ref{rem: standard:arguments}). Then arguing as in Lemma~\ref{lem: trans:nnf}, it suffices to check the no narrow fjords condition for $\varphi (x)$ as an element of the locally convex subset $\varphi(U) \cap R \subseteq \R^n$ (note that this set need not be closed whence it does not make sense to say that it has no narrow fjords!)

Now since $U \cap C$ is the intersection of an open and a closed subset of a locally compact space, it is locally compact, whence $\varphi(U\cap C) = \varphi (U) \cap R$ is locally compact. 
We can thus choose a compact neighborhood $K$ of $\varphi(x)$ in $\varphi (U) \cap R$.
Then local convexity of $\varphi (U) \cap R$ allows us to choose a neighborhood $\varphi(x) \in W \subseteq K$ which is convex. 
Denote now by $K_{\varphi(x)} = \overline{W} \subseteq K$ the closure of $W$. 
It is again a convex set and compact by construction. 
Now an easy but tedious computation involving metric estimates, convexity of $K_{\varphi(x)}$ and the boundary behavior observation $\partial (\varphi (U) \cap R) \cap K_{\varphi(x)} \subseteq \partial K_{\varphi(x)}$ yields the constants needed to verify the no narrow fjords condition for $K_{\varphi (x)}$. 
Alternatively, observe that convex sets have Lipschitz boundary (cf.\ again \cite{MR638797}) which implies that $K_{\varphi (x)}$ has no narrow fjords.

To check the outward polynomial cusp condition, we use again the compact, convex neighborhood $K_{\varphi(x)}$ of $\varphi(x)$ in $\varphi(C\cap U)$, and a diffeomorphism $\varphi(U) \simeq \R^n$ (which is locally bi-Lipschitz). Since the image of the compact convex neighborhood with Lipschitz boundary yields the required estimates for the polynomial outward cusp condition, the diffeomorphism transfers them from $\varphi(U)$ to $\R^n$ (and so satisfying Frerick's version of the definition). Invoking Lemma \ref{lem:transfer_cusp_condition}, we see that $ C\cap U $ satisfies the condition.
Hence $C$ has at worst polynomial outward cusps, completing the proof.
\end{proof}

For regular closed subsets which are at the same time submanifolds with rough boundary we prove now that for every vector bundle the space of sections from Definition~\ref{defn: sectsp:cl} can canonically be identified with the sections of the corresponding pullback bundle over the submanifold with rough boundary. 

\begin{proposition}\label{prop: quot:secsmfd}
Let $\pi_E \colon E \rightarrow M$ be a rank $k$ vector bundle over a manifold without boundary.
 Let $C \subseteq M$ be a closed submanifold with rough boundary. Assume in addition that 
 \begin{enumerate}
\item $C$ is a submanifold with corners, or
\item $C$ is compact.
\end{enumerate}
Then the pullback $p_C \colon \Gamma_c (M,E) \rightarrow \Gamma_c (C, \iota_C^* E),\ \sigma \mapsto \sigma \circ \iota_C$ by the inclusion $\iota_C \colon C \rightarrow M$ is a linear quotient map.
\end{proposition}

\begin{proof}
 Observe that since $C \subseteq M$ is closed, the inclusion $\iota_C$ is a proper mapping (i.e.\ preimages of compact sets are compact).
 Thus by our definition of submanifold with rough boundary (or with corners), every section  $\sigma \in \Gamma_c (M,E)$ induces a smooth pullback section $\sigma \circ \iota_C \in \Gamma_c (C,\iota_C^* E)$ and the pullback map $p_C$ makes sense and is clearly linear. 
Further, taking canonical identifications, a section $\tau \in \Gamma_c (C,\iota_C^* E)$ clearly coincides with a smooth mapping $\tilde{\tau}\colon C \rightarrow E$ with compact support such that $\pi_E \circ \tilde{\tau} = C$ (here we use $\tilde{\tau}$ to mark the difference in the codomain). 
Now Proposition \ref{prop: rsub:nnf} implies that $C$ has no narrow fjords, whence by Proposition \ref{prop: Whit:mfd} there is $\hat{\tau} \in \Gamma_c (M,E)$ which restricts to $\tilde{\tau}$ on $C$. We deduce that $p_C (\hat{\tau}) = \tau$, whence $p_C$ is surjective.

To establish continuity we have to distinguish the two cases, due to the difference in the function space topologies.
\begin{enumerate}
\item ($C$ is a submanifold with corners) The map $\iota_C^* \colon C^\infty_{\text{fS}} (M,E) \rightarrow C^\infty_{\text{fS}} (C,E), f \mapsto f\circ \iota_C$ is continuous by \cite[Theorem 7.3]{michor} (and even smooth) as $\iota_C$ is proper. 
Consider the linear subspace 
$\mathcal{D}_{\iota_C} (C,E) \coloneq \{g \in C^\infty (C,E) \mid \pi_E \circ g = \iota_C, g \equiv 0 \text{ off some compact set in } C\} $ of $C^\infty_{fS} (C,E)$.
It is easy to see that $\iota_C^*$ restricts to a continuous mapping 
$I \colon \Gamma_c (M,E) \rightarrow \mathcal{D}_{\iota_C} (C,E).$
However, due to the definition of the pullback bundle (see \cite[1.18 and 1.19]{michor}), the space $\mathcal{D}_{\iota_C} (C,E)$ is isomorphic as a linear and topological space to $\Gamma_c (C,\iota_C^* (E))$, composing $I$ with this isomorphism we obtain $p_C$ which is thus continuous.
\item ($C$ is compact) Since $M$ is $\sigma$-compact, we can choose and fix a locally finite (countable) atlas of bundle trivialisations $(W_i,\kappa_i)_{i\in \N}$ for $E$ such that every $W_i$ is relatively compact. Since $C$ is compact, only finitely many $W_i$ intersect $C$. After reordering, we may assume that $W_i \cap C \neq \emptyset$ iff $1\leq i \leq n$ for some $n\in \N$. 
Observe that since $W_i$ is an open subset of $M$ and $C$ is an embedded submanifold (with rough boundary), for $i\leq n$, the set $\iota_C^{-1} (W_i) = W_i \cap C$ is a submanifold with rough boundary of $W_i$---as submanifold charts $\varphi \colon U \rightarrow \varphi (U)$ are bijections, we find $\varphi(W_i \cap U \cap C) = \varphi (W_i \cap U) \cap R$. 
Hence Proposition \ref{prop: fcomp:cont} implies that 
\begin{align*}
p_i \colon C^\infty_{\text{co}} (W_i,\R^k) & \rightarrow C^\infty_{\text{co}} (W_i\cap C, \R^k)\\
f & \mapsto f|_{C \cap W_i}
\end{align*} is continuous linear for $i \leq n$. 
Thus we obtain a continuous (linear) map 
\begin{align*}
q \colon \bigoplus_{i \in \N} C^\infty_{\text{co}} (W_i, \R^d) &\rightarrow \bigoplus_{i \leq n}  C^\infty_{\text{co}} (W_i \cap C, \R^d)\\
(f_i)_i & \mapsto (p_i(f_i))_{1\leq i \leq n}.
\end{align*}
Following Lemma \ref{lem: top:sect} and \ref{setup: sect:loc} we obtain a commutative diagram 
\[
\xymatrix@C2em{
    \Gamma_c (M,E) \ar[r]^-\rho \ar[d]^{p_C} & \displaystyle\bigoplus_{i\in \N} \Gamma (W_i, E|_{W_i})  \ar[r]^\cong & \displaystyle \bigoplus_{i\in \N}  C^\infty_{\text{co}} (W_i,\R^d) \ar[d]^{q} \\
     \Gamma (C,E) \ar[r]^-\rho & \displaystyle\bigoplus_{1\leq i\leq n} \Gamma (W_i \cap C, (\iota_C^* E)|_{W_i \cap C}) \ar[r]^-\cong & \displaystyle \bigoplus_{1\leq i\leq n} C^\infty (W_i\cap C, \R^d)
    }
\]
where the the image the $\rho$ are topological embeddings with closed images. Hence $p_C$ is continuous in this case.
\end{enumerate}
Finally, let us establish that $p_C$ is a quotient map, i.e.\ $p_C$ is open.
To this end, recall from Appendix \ref{App: topo:funcsp} that $\Gamma_c (M,E)$ is an (LF)-space, i.e. webbed and ultrabornological. 
Further, if $C$ is a manifold with corners, also $\Gamma_c (C,\iota_C^* E)$ is an (LF) space. 
If $C$ is compact and a manifold with rough boundary, then $\Gamma (C,\iota_C^*E)$ is even a \Frechet space by \ref{setup: sect:loc}. 
In both cases, $p_C$ is open by the open mapping theorem \cite[24.30]{MR1483073}. 
\end{proof}

\begin{proposition}\label{prop: ident:sectsp}
Let $M$ be a manifold and $E \rightarrow M$ be a rank $k$ vector bundle.
If $C$ is closed submanifold with rough boundary which satisfies the assumptions of Proposition \ref{prop: quot:secsmfd}, then $\Gamma_c (C,E) = \Gamma_c (C,\iota_C^*E)$ as locally convex vector spaces.
\end{proposition}

\begin{proof}
Since $C$ is an embedded submanifold, a section in the pullback bundle is smooth if and only if it is a smooth as a mapping $C \rightarrow E$. Thus as sets we canonically identify $\Gamma_c (C,E) = \Gamma_c (C,\iota_C^*E)$.  
Now Proposition \ref{prop: quot:secsmfd} and Proposition \ref{prop: Whit:mfd} yield a commutative diagram 
\begin{displaymath}
\begin{xy}
\xymatrix{
    & \Gamma_c (M,E)\ar[ld]_{\res_C} \ar[rd]^{p_C} &\\   
    \Gamma_c (C,E) \ar[rr]^\id & &\ar@<1ex>[ll]^\id \Gamma_c (C,\iota_C^* E)}
\end{xy}
\end{displaymath} 
where the diagonal arrows are quotient mappings.
\end{proof}

\noindent
\textbf{Manifolds of mappings for non-compact source manifolds}\\
Assume that $C$ is a (sub-)manifold with corners which is possibly non-compact. 
Then the function space $C^\infty (C,N)$ can be endowed with an infinite-dimensional manifold structure which constructed similarly to the construction outlined in Section \ref{sect: mfdmap:rb}: One endows $C^\infty (C,N)$ withe the $\mathcal{FD}$-topology described in \cite{michor} (a Whitney type topology). 
In the boundaryless case \cite{HS17}, this topology is also called the fine very strong topology and therefore we denote by $C^\infty_{\text{fS}} (C,N)$ the function space with the fine very strong (=$\mathcal{FD}$-) topology.

Choosing a local addition on $N$, the construction of manifold charts is completely analogous to the construction outlined in \ref{setup: can:charts} with the notable exception that one has to restrict to $\Gamma_c (C,f^*TN)$ and one has to intersect $U_f$ with
$$\{g \in C^\infty (C,N) \mid \exists K \text{ compact, such that } \forall x \in C\setminus K\ f(x)=g(x)\}.$$
The rest of the construction is completely analogous to the one outlined in \ref{setup: can:charts} and yields the same structure as in Theorem \ref{thm: mm:rb} if $C$ is compact (note that we will thus also write $C^\infty_{fS} (C,N) = C^\infty_{\text{co}} (C,N)$ if $C$ is compact with rough boundary).

We are now ready to prove Theorem B which we restate here for the reader's convenience. Recall that $M$ is equipped with a Riemannian metric.

\begin{theorem}\label{thm: Bproof}
For $C\subset M$ a submanifold with corners, or compact and a submanifold with rough boundary, then the restriction map $\res_C^M \colon C^\infty_{\text{fS}} (M,N) \rightarrow C^\infty_{\text{fS}} (C,N)$ is a submersion of locally convex manifolds.
\end{theorem}

\begin{proof}
Let $\iota_C \colon C\rightarrow M$ be the canonical inclusion, which is smooth as $C$ is an embedded submanifold. Hence $\res_C^M=\iota_C^*$ is smooth by \cite[Theorem 7.3]{michor} (if $C$ is a submanifold with corners). 
Since the compact-open $C^\infty$ topology is coarser than the fine very strong topology (cf.\ \cite{HS17}), Proposition \ref{prop: fcomp:cont} implies that $\res_C^M$ is continuous if $C$ is compact and a submanifold with rough boundary (note that $M$ might be non-compact).
Hence to establish smoothness and the submersion property, it suffices to construct submersion charts for $\res_C^M$.

Let now $F\in C^\infty (M,N)$ and $f \coloneq \res_C^M (F)$. Then we use that $C$ satisfies the cusp condition and consider the canonical charts $(U_F,\varphi_F)$ and $(U_f,\varphi_f)$ (cf.\ \ref{setup: can:charts}) to obtain a commutative diagram
\begin{displaymath}
\xymatrix{
  C^\infty_{\text{fS}} (M,N)  \supseteq U_F \ar[d]_{\res_C^M}\ar[r]^-{\varphi_F} & \Gamma_c (M,F^*E) \ar[r]^-{\cong}_-{\text{Cor }\ref{splitting consequences}} \ar[d]^{\res_C}&  \mathcal{I}_c (C,M) \oplus \Gamma_c (C,F^*TN) \ar[d]^{\pr_2} \\ 
  C^\infty_{\text{fS}} (C,N) \supseteq U_f \ar[r]^-{\varphi_f} & \Gamma_c (C,f^*TN) \ar[r]^-\cong_-{\text{Prop } \ref{prop: ident:sectsp}} &  \Gamma_c (C,F^*TN)
    }
\end{displaymath}
Observe that $\res_C^M$ is a smooth submersion as the canonical charts conjugate it to a projection onto a complemented closed subspace, which is continuous  linear.
\end{proof}

Note that Corollary D from the introduction follows from the results in this section and Corollary C in the wash as Lemma \ref{lem: sconv:rbsmfd} asserts that strongly convex regular closed subsets of Riemannian manifolds are submanifolds with rough boundary.

\appendix 
 \section{ Essentials on infinite-dimensional calculus and function spaces}\label{app: calculus}
 In this appendix we collect the necessary background on the theory of manifolds
that are modelled on locally convex spaces and how spaces of smooth maps can be
equipped with such a structure. Let us first recall some basic facts concerning
differential calculus in locally convex spaces.
  \subsection*{Calculus in locally convex spaces}

 We base our investigation on the so called Bastiani calculus \cite{bastiani} and our exposition here follows \cite{MR1911979,MR2261066}. 

\begin{definition}\label{defn: deriv} Let $E, F$ be locally convex spaces, $U \subseteq E$ be an open subset,
$f \colon U \rightarrow F$ a map and $r \in \N_{0} \cup \{\infty\}$. If it
exists, we define for $(x,h) \in U \times E$ the directional derivative
$$df(x,h) \coloneq D_h f(x) \coloneq \lim_{t\rightarrow 0} t^{-1} \big(f(x+th) -f(x)\big).$$
We say that $f$ is $C^r$ if the iterated directional derivatives
\begin{displaymath}
d^{(k)}f (x,y_1,\ldots , y_k) \coloneq (D_{y_k} D_{y_{k-1}} \cdots D_{y_1}
f) (x)
\end{displaymath}
exist for all $k \in \N_0$ such that $k \leq r$, $x \in U$ and
$y_1,\ldots , y_k \in E$ and define continuous maps
$d^{(k)} f \colon U \times E^k \rightarrow F$. If $f$ is $C^\infty$ it is also
called smooth. We abbreviate $df \coloneq d^{(1)} f$ and for curves $c \colon I \rightarrow M$ on an interval $I$, we also write $\dot{c} (t) \coloneq \frac{\dd}{\dd t} c (t) \coloneq dc(t,1)$.
\end{definition}

We will frequently want to work with smooth mappings on non-open sets. Contrary to the treatment in the main body of the text we restrict ourselves here to regular sets which are locally convex.

\begin{definition}[Differentials
on non-open sets]\label{defn: nonopen}
\begin{enumerate}
\item A subset $U$ of a locally convex space $E$ is \emph{locally
convex} if every $x \in U$ has a convex neighborhood $V$ in $U$.
\item Let $U\subseteq E$ be a locally convex subset with dense interior and $F$ a locally convex space.
A continuous mapping $f \colon U \rightarrow F$ is called $C^r$ if
$f|_{U^\circ} \colon U^\circ \rightarrow F$ is $C^r$ and each of the
 $d^{(k)} (f|_{U^\circ}) \colon U^\circ \times E^k \rightarrow F$
admits a continuous extension
$d^{(k)}f \colon U \times E^k \rightarrow F$ (which is then necessarily
unique).
\end{enumerate}
\end{definition}

Note that for $C^r$-mappings on regular locally convex subsets the chain rule holds (whereas it becomes false in general without requiring local convexity, cf.\ Lemma \ref{lem: reschain}).
Hence there is an associated concept of locally
convex manifold with rough boundary.

 \subsection*{Topologies on function spaces with non-compact source}\addcontentsline{toc}{subsection}{Topologies on function spaces with non-compact source}\label{App: topo:funcsp}
 In this appendix we recall some basic facts on the topology of spaces of smooth sections in vector bundles over a non-compact manifold.
 For the rest of this section we let $M,N$ be finite-dimensional manifolds  and $p \colon E \rightarrow M$ be a vector bundle over $M$.
 Further, we denote by $\Gamma(M,E)$ the vector space of all smooth sections of the bundle and by $\Gamma_c (M,E) \subseteq \Gamma(M,E)$ the space of compactly supported sections.
 
 \begin{setup}\label{setup: topo:open}
 For the space of smooth mappings between manifolds with corners $C^\infty (M,N)$ we consider the so called $\mathcal{FD}$-topology or \emph{fine very strong topology} and write $C^\infty_{\text{fS}} (M,N)$ for the space endowed with this topology.
 This is a Whitney type topology controlling functions and their derivatives on locally finite families of compact sets. Before we describe a basis of the fine very strong topology, we have to construct a basis for the strong topology which we will then refine. To this end, we recall the construction of the so called basic neighborhoods (see \cite{HS17}). Consider $f$ smooth, $A$ compact, $\varepsilon >0$ together with a pair of charts $(U, \psi)$ and $(V,\varphi)$ such that $A \subseteq V$ and $\psi \circ f \circ \varphi^{-1}$ makes sense.
 Then we use multiindex notation \ref{defn: multiindex} to define an \emph{elementary $f$-neighborhood} $\mathcal{N}^{r} \left( f; A , \varphi,\psi,\epsilon \right)$ as
 $$
\left\{\substack{\displaystyle g \in C^\infty (M,N), \quad \psi \circ g|_A \quad\text{ makes sense,}\\ \displaystyle\sup_{\alpha \in \N_0^d, |\alpha|\leq r } 
 \sup_{x \in \varphi (A)}\lVert \partial^\alpha \psi \circ f \circ \varphi^{-1}(x) - \partial^\alpha\psi \circ g \circ \varphi^{-1}(x)\rVert < \varepsilon}\right\}.
 $$ 
 A basic neighborhood of $f$ arises now as the intersection of (possibly countably many) elementary neighborhoods $\mathcal{N}^{r} \left( f; A_i , \varphi_i,\psi_i,\epsilon_i \right)$ where the family $(V_i,\varphi_i)_{i\in I}$ is locally finite. We remark that basic neighborhoods form the basis of the very strong topology (see \cite{HS17} for more information). To obtain the fine very strong topology, one declares the sets  
 \begin{equation}\label{eq: cpt:nbhd}
 \{g \in C^\infty (M,N) \mid \exists K \subseteq M \text{ compact, s.t. } \forall x \in M\setminus K,\ g(x) =f(x) \} \tag{$\star$}
 \end{equation}
 to be open and constructs a subbase of the fine very strong topology as the collection of sets $\eqref{eq: cpt:nbhd}$ (where $f \in C^\infty (M,N)$) and the basic neighborhoods of the very strong topology. 
   
 Note that in \cite{HS17} the fine very strong topology was only considered for manifolds without boundary (and coincides with the $\mathcal{FD}$-topology, see \cite[Appendix C]{HS17}). For manifolds with corners, we refer to \cite{michor} for more information on this topology. 
  
 If $M$ is compact, all topologies mentioned above coincides with the compact open $C^\infty$-topology from Definition \ref{defn: topo:init}. Further, the fine-very strong topology turns $C^\infty (M,N)$ into an infinite-dimensional manifold (cf.\ \cite{michor} and \cite{HS17}).
 If $N = \R^n$ then the pointwise operations turn $C^\infty_{\text{fS}} (M,\R^n)$ into a locally convex vector space (which is in not a \Frechet space if $M$ is not compact).
 \end{setup}

We now turn to the space of compactly supported sections of a vector bundle.

\begin{setup}{Compactly supported sections of a vector bundle.}\label{setup: cs}
Let $p \colon E \rightarrow M$ be a finite rank vector bundle over the finite dimensional manifold $M$ (possibly with corners). We consider three spaces of sections 
\begin{align*}
\Gamma (M,E) &\coloneq \{f \in C^\infty (M,E) \mid p \circ \sigma = \id_M\},\\
\Gamma_K (M,E) &\coloneq \{f \in \Gamma (M,E) \mid \text{supp } f \subseteq K \text{for } K \subseteq M \text{ compact}\},\\
 \Gamma_c (M,E) &\coloneq \bigcup_{K \subseteq M \text{ compact}} \Gamma_K (M,E). \end{align*}
Endow $\Gamma (M,E)$, $\Gamma_K (M,E)$ and $\Gamma_c (M,E)$ with the subspace topology from $C^\infty_{\text{fS}} (M,E)$; we obtain locally convex vector spaces \cite[Proposition 4.8 and Remark 4.11]{michor}.
Moreover, we remark that by compactness of $K$, the topology on $\Gamma_K (M,E)$ coincides with the subspace topology induced by the compact open $C^\infty$-topology, whence one can prove that $\Gamma_K (M,E)$ is a \Frechet space (cf.\ e.g.\ \cite[Section 3.1]{MR1934608}). 
Further, by \cite[Proposition 4.8 and Remark 4.11]{michor} $\Gamma_c (M,E)$ is the inductive limit (in the category of locally convex spaces) of the \Frechet spaces $\Gamma_K (M,E)$, where $K$ runs through a compact exhaustion of $M$. Thus $\Gamma_c (M,E)$ is an (LF)-space.
\end{setup}  

If $M$ is a manifold without boundary, a different description of the topology on $\Gamma_c (M,E)$ is considered in \cite[Appendix F]{math/0408008v1}.
By the following lemma, this topology on \( \Gamma_c (M,E) \) is equivalent to the fine very strong subspace topology.
\begin{lemma}\label{lem: top:sect}
  If $M$ is a manifold without boundary, the following describe equivalent locally convex topologies on \( \Gamma_c (M,E) \):
  \begin{enumerate}
  \item Give \( \Gamma_c (M,E) \subseteq C^\infty_{\text{fS}}(M,E) \) the subspace topology.
    \item Let \( \cU =\left\lbrace U_i \right\rbrace_{i \in I} \) be a locally finite cover of \( M \) by relatively compact open subsets \( U_i \subseteq M \).
      Equip each \( \Gamma(U_i,E|_{U_i}) \subseteq C^\infty_{\text{co}} (U_i,E|_{U_i}) \) with the subspace topology, and give \( \Gamma_c (M,E) \) the topology induced by
      \begin{align*}
        \rho_{\cU} \colon \Gamma_c (M,E) \to \bigoplus_{i \in I} \Gamma(U_i,E|_{U_i}),
      \end{align*}
      where the direct sum is given the box topology, and \( \rho_i \colon \Gamma_c (M,E) \to \Gamma(U_i,E|_{U_i}) \) is the restriction map.
      \item For \( K \subseteq M \) compact, give 
      \begin{align*}
        \Gamma_K(M,E) = \left\lbrace s \in \Gamma(M,E) : \operatorname{supp} s \subseteq K \right\rbrace \subseteq C^\infty (M,E)_{\text{co}}
      \end{align*}
      the subspace topology.
      Now equip \( \Gamma_c (M,E) \) with the final locally convex vector space topology with respect to the inclusion maps $\iota_K \colon \Gamma_K(M,E) \to \Gamma_c (M,E)$ as \( K \) ranges through the compact subsets of \( M \).
  \end{enumerate}
\end{lemma}

\begin{proof}
  Note that the topology described in 3.\ is the locally convex inductive limit topology induced by the inductive system $\{\iota_K\}_{K \subseteq M \text{ is compact}}$ (ordered by inclusion of compact sets). Hence the topologies described in 1.\ and 3. coincide by \ref{setup: cs}. However, also the topologies 2.\ and 3.\ coincide by \cite[F.19]{math/0408008v1}.
\end{proof}

\begin{remark}\label{rem: loc:sect}
Assume that the sets $U_i$ from Lemma \ref{lem: top:sect} 2. are domains of manifold charts $\varphi_i \colon U_i \rightarrow \R^m$ of $M$. Then we define for each $X \in \Gamma(U_i,E|_{U_i})$ the local representative
$$X_{\varphi_i} \coloneq \pr_2 \circ T\varphi_i \circ X \varphi_i^{-1} \in C^\infty (\varphi_i (U_i) , \R^m),$$
where $\pr_2 \colon \varphi_i (U_i) \times \R^m \rightarrow \R^m$ is the canonical projection. This mapping yields an isomorphism of locally convex spaces
$$\Gamma (U_i,E|_{U_i}) \rightarrow C^\infty_{\text{co}} (\varphi_i (U_i),\R^m),\quad  X \mapsto X_{\varphi_i}.$$
\end{remark}

\section{ The space of (smooth) Whitney jets}\label{app: Whitney}

In this appendix we recall some details from Whitney's approach to the the extension problem for smooth functions on a closed subset of $\R^n$. Though the exposition in the main part of the article does not need these results as such (since we will only cite their consequences from \cite{MR2300454}); the authors think that a quick recollection of these constructions will be beneficial to understand the underlying ideas. Our exposition follows here \cite[Section 2]{MR2300454}; we mention that a more in-depth treatment can be found in \cite[Chapter 2]{MR2882877}.

 \begin{setup}\label{defn: multiindex}
 Throughout this section we use standard multiindex notation,\\ i.e.\ $\alpha = (\alpha_1,\alpha_2,\ldots, \alpha_d) \in \N_0^d$, $|\alpha| = \sum_{i=1}^d \alpha_i$, $\alpha! = \alpha_1! \cdots \alpha_d!$ and $\partial^\alpha = \partial_1^{\alpha_1} \cdots \partial_d^{\alpha_d}$.
 \end{setup}
 
 \begin{definition}
  Let $K \subseteq \R^d$ be compact. For a family $f = (f^\alpha) \in \prod_{\alpha \in \N^d_0} C(K)$ and $x \in K$ we define the formal Taylor polynomial 
  $$\Tay_x^m f (y) \coloneq \sum_{|\alpha|\leq m} \frac{f^{\alpha}(x)}{\alpha!} (y-x)^\alpha. $$
The formal Taylor remainder $R^m_x f \in \prod_{|\alpha| \leq m} C(K)$ is then defined by 
  $$(R_x^m f)^\alpha \coloneq f^\alpha - \partial^\alpha (\Tay_x^{m}f)|_{K} = f^\alpha - \Tay^{m-|\alpha|}_x (f^{\alpha +\beta})_{|\beta| \leq m -\alpha}|_K,\ |\alpha| \leq m.$$ 
  \end{definition}
  
  \begin{setup}\label{setup: remainder:id}
    For $|\alpha| < m$ one easily checks the identities 
   \begin{align*}
   \partial^\alpha \Tay_x^m f = \Tay_x^{m-|\alpha|} ((f^{\alpha+\beta})_{\beta \in \N_0^{m -|\alpha|}}) \\
   (R_x^m f)^\alpha = (R_x^{m-|\alpha|} (f^{\alpha +\beta})_{\beta \in \N_0^{m -|\alpha|}})^0
   \end{align*}
  \end{setup}
  
  We now define seminorms on the spaces of jets which allow us to define a \Frechet topology on the space of Whitney jets (cf.\ Definition \ref{defn: Whitneyjet} below).
  
  \begin{definition}
  For $f = (f^\alpha) \in \prod_{\alpha \in \N_0^d} C(K)$ and $m \in \N_0$ we define the seminorms
  \begin{align*}
  |f|_{m,K} \coloneq \sup_{x \in K} \sup_{|\alpha| \leq m} |f^\alpha (x)|.
  \end{align*}
  \end{definition}
  
  The seminorms $|\cdot|_{m,K}$ are closely connected to the compact open $C^\infty$-topology on the space $C^\infty (U, \R)$ as we will discuss in \ref{Whitney:fun} below. 
    \begin{definition}
 For $f = (f^\alpha) \in \prod_{\alpha \in \N_0^d} C(K)$, $m \in \N_0$ and $t>0$, we define 
  \begin{align*}&q_m (f,t) \coloneq q_m (K,f,t)\\ \coloneq& \sup \{|R_x^m (f)^\alpha (y)||y-x|^{|\alpha| - m} \colon x,y \in K, 0 < |x-y|\leq t, |\alpha|\leq m\}
  \end{align*}
  and 
  \begin{align*}
  \lVert \cdot\rVert_{m,K} \coloneq |\cdot|_{m,K} + \sup_t q_m (K, \cdot ,t)
  \end{align*}
  \end{definition}
  
  \begin{setup}\label{Whitney:fun}
 Let $U \subseteq \R^d$ be open and $K_1 \subseteq K_2 \subseteq \ldots$ a fundamental sequence of compact sets.\footnote{i.e.\ for all $n\geq 1$, $K_n \subseteq K_{n+1}^\circ$ and $U = \bigcup_l K_l$.} 
  Then we recall from Remark \ref{rem: coinfty} that the compact open $C^\infty$-topology is initial with respect to the map 
  $$\partial \colon C^\infty (U,\R) \rightarrow \prod_{\alpha \in \N_0^d} C(U),\quad f \mapsto (\partial^\alpha f).$$
 Pulling back the seminorms on the jet space by the mappings $\partial (\cdot)|_{K_l}$, we obtain two families of seminorms:
  \begin{enumerate}
  \item $\{|\partial (\cdot)|_{K_l}|_{m,K_l} \mid l \in \N, m \in \N_0\}$. These seminorms are the classical seminorms which induce the compact open $C^\infty$-topology (cf.\ e.g.\ \cite{MR1934608,MR2261066}).
  \item  $\{\lVert \partial (\cdot)|_{K_l} \rVert_{m,K_l} \mid l \in \N, m\in \N_0\}$. Also these seminorms induce the compact open $C^\infty$-topology (this is easily seen by the usual estimate for the $m$th Taylor remainder using the $m+1$st derivative on closed balls covering $K_l$) (see e.g.\ \cite[\S2]{MR2300454}).
  Further we notice that $f \in \prod_{|\alpha| \leq m} C(U)$ is contained in the image of $\partial$ if and only if $\lim_{t \rightarrow 0} q_m (K,f,t)=0$ for all $m\in \N_0$ and $K\subseteq U$ compact.
  \end{enumerate}
 \end{setup}
 
 These considerations lead to the following definition.
 
 \begin{definition}\label{defn: Whitneyjet}
 Let $K \subseteq \R^d$ be compact. We say $f \in \prod_{\alpha \in \N_0^d} C(K)$ is a \emph{Whitney jet (of order $\infty$)}, if $\lim_{t \rightarrow 0} q_m (K,f,t)=0$ for all $m \in \N$.
 We denote the \emph{\Frechet space of all Whitney jets (of order $\infty$)} equipped with the seminorms $\lVert \cdot\rVert_m, m\in \N$ by $\mathcal{E}(K)$.
 \end{definition}

    Since smooth functions on an open set $U$ restrict to Whitney jets on every $K \subseteq U$ compact (cf.\ \cite[p.\ 125]{MR2300454}), we have (as sets)
   \begin{equation}\label{eq: projlim}
    C^\infty (U,\R) = \proj_n \mathcal{E}(K_n)
   \end{equation}
   for every fundamental sequence of compact sets $K_1 \subseteq K_2\subseteq \ldots \subseteq U$.
   However, we can also view this limit in the category of locally convex spaces and it follows from Remark \ref{Whitney:fun} that the locally convex topology on the left hand side of \eqref{eq: projlim} coincides with the compact-open $C^\infty$-topology described in Definition \ref{defn: topo:init} (cf.\ Remark \ref{rem: coinfty}).
   
 \begin{definition}\label{defn: closed}
 Let $U \subseteq \R^d$ be an open set with $(K_l)_{l\in \N}$ a fundamental sequence of compact sets and $F \subseteq U$ a closed set.
 Then we define the \emph{space of Whitney jets on $F$} as the projective limit
 $\mathcal{E}(F) \coloneq \proj_{l \in \N} \mathcal{E}(F\cap K_l)$ (in the category of locally convex spaces).
 \end{definition}

Note that the definition of $\mathcal{E}(F)$ is independent of the choice of fundamental sequence used to define it. 
 
 Using the canonical identification, we have $C^\infty_{\text{co}}  (U,
 \R^m) \cong C^\infty_{\text{co}} (U,\R)^m$ as \Frechet spaces. Hence it makes sense to define Whitney jets of vector valued functions as the Whitney jets of the components of the functions:
 
 \begin{definition}\label{defn: WJm}
 Define $\mathcal{E}(K,\R^m)$ to be the closed subspace of the product $\prod_{\alpha \in \N_0^d} C(K,\R^m)$ which corresponds to $\mathcal{E}(K)^m$ under the identification $\prod_{\alpha \in \N_0^d} C(K,\R^m) \cong (\prod_{\alpha \in \N_0^d} C(K))^m$.  
 Similarly for $(K_l)_l$ and $F$ as in Definition \ref{defn: closed} we define $\mathcal{E}(F,\R^m) \coloneq \proj_{l\in \N} \mathcal{E} (K_l,\R^m)$.
 \end{definition}

\nocite{*}
\bibliographystyle{cdraifplain}
\bibliography{extension}
\end{document}